\documentclass[10pt]{amsart}
\usepackage[all]{xy}
\usepackage{amsmath,amscd,amsfonts,amsthm,amssymb,latexsym,mathrsfs}
\usepackage{graphicx,epsfig}
\usepackage{url,hyperref}
\usepackage[usenames,dvipsnames]{color}

\DeclareMathOperator{\im}{{\rm im}}

\DeclareMathOperator{\cork}{{\rm cork}}
\DeclareMathOperator{\Frac}{{\rm Frac}}

\newcommand{\val}{{\rm val}}

\DeclareMathOperator{\conv}{{\rm conv}}
\DeclareMathOperator{\cone}{{\rm cone}}

\DeclareMathOperator{\length}{{\rm length}}
\DeclareMathOperator{\rec}{{\rm rec}}

\DeclareMathOperator{\Ann}{{\rm Ann}}
\DeclareMathOperator{\maxSpec}{\operatorname{maxSpec}}
\newcommand{\Mat}[1]{\text{\rm $#1$-Mat}}

\newcommand{\Gr}{Gr}
\newcommand{\Fl}{Fl}

\newcommand{\K}{\mathbb K}
\newcommand{\N}{\mathbb N}
\newcommand{\Z}{\mathbb Z}

\newcommand{\R}{\mathbb R}

\renewcommand{\P}{\mathbb P}

\newcommand{\m}{\mathfrak m}
\newcommand{\Trop}{\operatorname{Trop}}

\let\oldmarginpar\marginpar
\renewcommand\marginpar[1]{\-\oldmarginpar[\footnotesize #1]{\raggedright\footnotesize #1}}

\newtheorem{theorem}{Theorem}[section]
\newtheorem{lemma}[theorem]{Lemma}
\newtheorem{proposition}[theorem]{Proposition}
\newtheorem{corollary}[theorem]{Corollary}

\theoremstyle{definition}
\newtheorem{definition}[theorem]{Definition}
\newtheorem{Example}[theorem]{Example}
\newenvironment{example}[1][]{\begin{Example}[#1]\pushQED{\qed}}{\popQED\end{Example}}

\theoremstyle{remark}
\newtheorem{Remark}[theorem]{Remark}
\newenvironment{remark}[1][]{\begin{Remark}[#1]\pushQED{\qed}}{\popQED\end{Remark}}
\newtheorem{Question}[theorem]{Question}
\newenvironment{question}[1][]{\begin{Question}[#1]\pushQED{\qed}}{\popQED\end{Question}}

\numberwithin{equation}{section}

\begin{document}
\title[Polyhedra and parameter spaces for matroids over valuation rings]{Polyhedra and parameter spaces\\for matroids over valuation rings}
\author{Alex Fink \and Luca Moci}
\thanks{The first author was supported by EPSRC grant EP/M01245X/1.}

\maketitle

\begin{abstract}
In this paper we address two of the major foundational questions in the theory of matroids over rings. 
First, we provide a cryptomorphic axiomatisation, by introducing an analogue of the base polytope for matroids. Second, we describe a parameter space for matroids over a valuation ring, which turns out to be 
a tropical version of the Bott-Samelson varieties for the full flag variety. Thus a matroid over a valuation ring is a sequence of flags of tropical linear spaces a.k.a.\ valuated matroids.
\end{abstract}

\section{Introduction}

\emph{Matroids over rings} generalise several combinatorial constructions, such as matroids, \emph{valuated matroids} \cite{DW} and \emph{(quasi-)arithmetic matroids} \cite{Branden-Moci, D'Adderio-Moci}.
Since their recent introduction in \cite{FM}, they proved to be a powerful language with a wide number of applications, ranging from algebraic statistics \cite{KKL} to combinatorial topology \cite{BBGM, DeMo} and discrete geometry \cite{DeRi}.
Just as the notion of a matroid axiomatises the linear dependencies of a list of vectors, the notion of a matroid
over a ring $R$ axiomatises the relations of a list of elements in a $R$-module.
When $R$ is a Dedekind ring, several constructions which are crucial in matroid theory can be extended to matroids over $R$, such as duality and a universal deletion-contraction invariant \cite{FM}.

However, a significant lacuna was the lack of \emph{cryptomorphisms} (i.e.\ equivalent axiomatisations) for matroids over $R$.
Here we provide one of these, the analogue of the \emph{base polytope}, in Section~\ref{sec:P}.
Introduced in~\cite{Edmonds,GGMS}, the base polytope of a classical matroid is the convex hull of the indicator functions of the bases. Our analogue is not bounded (in fact it is a polyhedral cone) and lives in a higher dimensional ambient space, but still shares many of the theory's nice features,
e.g.\ it is well behaved under matroid operations (Proposition~\ref{prop:P ops}).
Moreover, polytopes arising from a matroid over $R$ can be fully characterised, mostly in term of their edge directions
(Theorem~\ref{thm:polyhedron}), which provides our cryptomorphism.
Martino~\cite{Martino1,Martino2} has provided further cryptomorphisms
when the base ring is~$\mathbb Z$,
filling the role of an \emph{independent set complex}.

The approach developed in \cite{FM} was based on \emph{localising} a matroid over a Dedekind ring $R$ at each prime ideal of $R$, hence obtaining a suite of matroids over discrete valuation rings.
We continue this study in the present paper,
though we broaden to the case of \emph{valuation rings}, not necessarily discrete.
This includes the case of Puiseux series, of interest in tropical geometry.
The global conditions on a matroid over a Dedekind domain are not stringent once its localisations are known,
so our focus does not discard a great deal of structure and we can define a polyhedron in this case as well.
However, these polyhedra can no longer be recognised by the directions of their edges, as Example~\ref{ex:two primes} shows.
In Sections \ref{sec2} and~\ref{sec3} we characterise matroids over a valuation ring,
in terms of a family of inequalities that can all be obtained from the tropicalisation of one Pl\"ucker relation for the Grassmanian, by two logical operations that we call \emph{genericisation} and \emph{zeroisation}.
The description is particularly simple and rich for a wide class of matroids that we call \emph{spannable} (see Section~\ref{sec:spannable}).

Furthermore we answer, in Sections~\ref{sec:TLS} and \ref{sec:parameter},
the questions of~\cite{FM} regarding what data inheres in a matroid over a valuation ring,
beyond the flag of \emph{tropical linear spaces} (a.k.a. {valuated matroids}) noted there.
Our answer, Theorem~\ref{thm:BS},
is phrased in the language of a \emph{parameter space} for matroids over valuation rings.
This parameter space, which is the intersection of the \emph{Dressian} with a linear space (Proposition \ref{prop:slice}),
is closely related to a tropical analogue of type~A \emph{Bott-Samelson varieties}.
These arise in algebraic geometry as resolutions of singularities of Schubert varieties
that behave very tractably on account of their iterated bundle structure and amenability to toric methods;
beyond the study of the Schubert varieties themselves, e.g.~\cite{BK05},
they find applications to topics such as 
representation theory of reductive algebraic groups \cite{Demazure},
Schubert and related polynomials \cite{Magyar}, 
associahedra and more generally brick polytopes \cite{Escobar},
and rhombus tilings \cite{rhombus}.
For our connection to be exact there are technical adjustments to be made:
intersecting with a cone, reinserting data lost by projectivisation,
and taking the direct limit over a system of inclusions.
Overlooking these adjustments, a matroid over a valuation ring is a \emph{sequence} of flags
of valuated matroids, starting with the flag observed in~\cite{FM}
and changing one step at a time to arrive at the standard flag
where all valuations are trivial.

\subsection{Comparison of matroid generalisations}
Matroids over rings differ in an apparent way from the ``orthodox'' picture of generalised matroids
which has flourished in the last years, represented by
matroids over hyperfields~\cite{Baker16},
over fuzzy rings~\cite{Dress1} (these last two having been unified in~\cite{GJ}),
and over partial fields~\cite{SW}.
All of these objects set about generalising the coefficients in a realisation:
so when the base is made to be a field $\mathbb K$, what comes out are matroids realised over~$\mathbb K$,
and to capture simply the set of all classical matroids, one needs a new base ``simpler'' than any field.
By contrast, we generalise the values taken by the rank function from natural numbers to modules:
matroids over any field in our sense are classical matroids,
and the choice of the field doesn't matter (any field having equally trivial module theory).

The present work begins to bridge the gap between these two worlds of generalisations,
by casting the entirety of the structure of matroids over a valuation ring $R$ in terms of valuated matroids.
Valuated matroids belong to the world of generalised coefficients:
they are matroids over the tropical hyperfield, or over the valuated fuzzy ring.
This bridge can be lifted to matroids over a Dedekind domain,
by constructing a suitable product of tropical hyperfields, one for each prime.

This is not enough to allow matroids over~$R$
to be directly interpreted as matroids over some hyperfield (vel sim.),
because their structure is richer: in fact the data in a matroid over a hyperfield belongs only to subsets of the ground set of fixed size,
while that in a matroid over~$R$ belongs to all subsets.
However, Corollary~\ref{cor:span cospan} shows that \emph{bispannable} matroids over~$R$,
those which are spannable and have spannable duals---%
the situation corresponding to a configuration of elements \emph{generating} a \emph{free} $R$-module---%
are determined by the modules on sets of a single size.
Thus bispannable matroids over~$R$ may literally be matroids over a suitable hyperring (Question~\ref{q:hyperfield}).

We should not omit to discuss arithmetic matroids here.
Like matroids over~$\mathbb Z$, arithmetic matroids are meant to encode dependences among a list of elements in a finitely generated abelian group.
The motivation for their introduction was the theory of arrangements of codimension~1 subtori in a torus,
in particular describing the Poincar\'e polynomials of complements of such arrangements.
Thus arithmetic matroids have an axiom with no counterpart in matroids over~$\mathbb Z$
ensuring, roughly, that the numbers counting some set of \emph{layers} of the arrangement come out nonnegative.
On the other hand, arithmetic matroids do not care about the tropical linear space structure inherent in matroids over $\Z$.
See Remark~\ref{rem:arithmetic} for more.

\subsection{Conventions}\label{ssec:conventions}
Zero is a natural number.
The minimum of an empty set is $\infty$.

If $G$ is an ordered cancellative abelian monoid,
then we take $\infty+\ell=\infty-\ell=\infty$
for all $\ell\in G\cup\{\infty\}$.
Perhaps the only surprise here is that $\infty-\infty = \infty$.
Setting this convention allows us to convert without fuss between
nondecreasing sequences of elements of $G\cup\{\infty\}$ and their sequences of partial sums,
which is an operation we use a lot when handling the quantities $t_i$ of Proposition~\ref{prop:t}.

Polyhedra are always convex but need not be bounded.  Polytopes are bounded polyhedra.
We write $\cone S$ for the cone generated by a set of vectors $S$.
If $P$ is a polyhedron we write
$\rec P=\{v : P+v \subseteq P\}$ for its \emph{recession cone},
and $\tan_v P=\bigcup_{\lambda\geq0}\lambda(P-v)$ for its \emph{tangent cone} at the point~$v$,
which is a cone over the origin.

We augment subsets of a ground set by one or a few elements throughout this paper,
so we will use streamlined notation for this operation.
Given a subset $A$ of a set $E$ and elements
$b,c,\ldots\in E\setminus A$, we will write $Ab$ for $A\cup \{b\}$, $Abc$ for $A\cup \{b,c\}$, and so on.

\subsection*{Acknowledgments}
The authors thank Michael Joswig and Benjamin Schr\"oter for useful discussions,
and the latter for a close reading.

\tableofcontents

\section{Modules over valuation rings and their invariants}\label{sec2}

We first recall the definition of matroids over rings, quoted from~\cite{FM}.

\begin{definition}\label{def:matroid}
Let $R$ be a commutative ring.
A \emph{matroid over $R$} on the ground set $E$
is a function $M$ assigning to each subset $A\subseteq E$
a finitely generated $R$-module $M(A)$ such that
\begin{itemize}
\item[(M)] for every subset $A\subseteq E$ and elements
$b,c\in E$, there exist elements
$x=x(b,c)$ and $y=y(b,c)$ of $M(A)$ satisfying
\begin{align*}
M(A\cup\{b\}) &\cong M(A)/(x) \\
M(A\cup\{c\}) &\cong M(A)/(y) \\
M(A\cup\{b,c\}) &\cong M(A)/(x,y).
\end{align*}
\end{itemize}
\end{definition}

The principal, but not the only, class of examples is the \emph{realisable} matroids over~$R$,
in which $M(\emptyset)$ is an arbitrary (finitely generated) $R$-module,
$x_e$ is an element of $M(\emptyset)$ for each $e\in E$,
and $M(A) = M(\emptyset) / \langle x_e : e\in A\rangle$.

In fact, all our matroids over~$R$ will be \emph{finitely presented},
i.e.\ they will be such that every module $M(A)$ is finitely presented.
When $R$ is Noetherian, as it was for all the substantial results of~\cite{FM},
being finitely presented is equivalent to being finitely generated.%

Throughout this paper, with the single exception of Proposition~\ref{prop:Dedekind poly} and the ensuing discussion,
$R$ will be a fixed valuation ring with maximal ideal $\m$
and valuation $\val$.  We straightforwardly write $\val(R)$ for the range of the valuation,
an ordered cancellative abelian monoid containing the absorbing element $\infty=\val(0)$ and the identity $0=\val(1)$.
We write $\val(R)^+$ to mean $\val(R)\setminus\{0\}$.
From Section~\ref{sec:P} onward
we further assume that $R$ is of height one, so that its valuation can be taken to be
$\val:R\to\R\cup\{\infty\}$. 
This assumption is taken for simplicity, but is probably inessential.
Tropical geometry for fields with more general valuations is particularly underexplored:
\cite{Aroca} is a first investigation.

We will not need assumptions on~$R$ beyond these to study the structure of matroids over~$R$.
It is possible to work with the definition if one knows only
the classification of finitely presented $R$-modules up to isomorphism
and the knowledge of which isomorphism types admit surjections with cyclic kernel
or pushout squares of such surjections.
The structure theory of modules is rigid enough
that the answers to these questions depend only on $\val(R)$, as we will see.

Discrete valuation rings, for which
$\val(R)=\N\cup\{\infty\}$, are examples of valuation rings of height one.
Another example is the ring of integers of the \emph{Puiseux series} over a field $\K$,
namely the ring of formal sums of the form
\[x = \sum_{q\in Q} a_qt^q\]
in an indeterminate $t$ where $a_q\in K\setminus\{0\}$ and
$Q$ is a well-ordered set of positive rational numbers,
under the valuation taking this element $x$ to~$\min Q$.

Finitely generated ideals of~$R$ are principal, and therefore all
finitely presented quotients of~$R$ are of the form $R/(x)$ for $x\in R$.
Every finitely presented $R$-module $N$ is a direct sum of these quotients
\cite[Theorem 1]{Warfield}.
The Grothendieck group of finitely-presented $R$-modules containing no submodule isomorphic to~$R$
is isomorphic to the value group of~$R$, under the isomorphism
sending $R/(x)$ to $\val(x)$ for $x\in R$.
We define the \emph{length} of an $R$-module to be $\infty$ if it contains a
submodule isomorphic to~$R$, and its image in the value group of~$R$ above otherwise.
For instance, if $R$ is a discrete valuation ring with $\val(R)=\mathbb N$
and maximal ideal $\mathfrak m$, then $\length(N) = \dim_{R/\m}N$
for any $R$-module $N$.

Given a finitely presented $R$-module $N$, define a family of invariants
\[t_i(N) = \min_{x_1,\ldots,x_i\in N}\length(N/\langle x_1,\ldots,x_i\rangle)\]
indexed by the natural numbers $i\geq0$.
The collection of all the $t_i(N)$ is a complete isomorphism invariant
of the module $N$, because if
\[N \cong R/I_1\oplus\cdots\oplus R/I_s\]
with $0\subseteq I_1\subseteq\cdots\subseteq I_s\subsetneq R$ principal ideals,
then taking $x_1,\ldots,x_i$ to be generators of the first $i$ of these
summands achieves the minimum length of~$N/\langle x_1,\ldots,x_i\rangle$.
Therefore, by our conventions about differences involving $\infty$, we have
$\length(R/I_i) = t_{i-1}(N) - t_i(N)$,
implying that
\[I_i = \{r\in R : \val(r)\geq t_{i-1}(N) - t_i(N)\}.\]

In just the same way, a matroid $M$ over~$R$ on ground set~$E$ is uniquely determined by
the collection of all the $t_i(A)$ for $i\geq0$ and $A\subseteq E$.
When there is a privileged matroid $M$ over~$R$ in the context,
we will abbreviate $t_i(A)\doteq t_i(M(A))$.
Our first objective is to reframe
the characterisation of these matroids from Section~5 of~\cite{FM}
in terms of these quantities.

\begin{proposition}\label{prop:t}
Let $t_i(A)$ be an element of $\val(R)\cup\{\infty\}$ for each subset $A\subseteq E$
and each natural $i\geq0$.  The $t_i(A)$ are the data of a matroid $M$ over~$R$
if and only if the following hold:
\begin{enumerate}
\item[(TS)] for all $A\subseteq E$, the sequence $(t_i(A))_{i\in\N}$ stabilises at zero;
\item[(T0)] for all $A\subseteq E$ and $i\geq0$,
\[t_i(A)-t_{i+1}(A)\geq t_{i+1}(A)-t_{i+2}(A);\]
\item[(T1)] for all $A\subseteq E$, $b\in E\setminus A$, and $i\geq0$,
\[t_i(A)-t_{i+1}(A)\geq t_i(Ab)-t_{i+1}(Ab)\geq t_{i+1}(A)-t_{i+2}(A);\]
\item[(T2)] for all $A\subseteq E$, $b,c\in E\setminus A$, and $i\geq0$,
\[t_{i+1}(A)-t_{i+1}(Ab)-t_{i+1}(Ac)+t_i(Abc)\geq
\min\{t_i(Ab)-t_{i+1}(Ab), t_i(Ac)-t_{i+1}(Ac)\},\]
and equality is attained if $t_i(Ab)-t_{i+1}(Ab)\neq t_i(Ac)-t_{i+1}(Ac)$.
\end{enumerate}
\end{proposition}

With the results of~\cite{FM} as starting point,
there are two parts to the proof of Proposition~\ref{prop:t}.
One is simply to replicate those results in the setting of a non-discrete
valuation ring.
For $N$ a finitely presented $R$-module, define the quantities
\begin{align*}
d_\ell(N) &= \min\{i\in\Z : t_i(N) - t_{i+1}(N) < \ell\}\in\Z, \\
d_{\leq\ell}(N) &= \mbox{the K-class of $N/I_\ell N$} \in\val(R).
\end{align*}
for each $\ell\in\val(R)^+$.
Observe that, as $\ell$ varies, the quantities $d_\ell(N)$ completely determine $N$
and therefore determine the quantities $d_{\leq\ell}(N)$.
We use this in the next lemma, silently treating the $d_{\leq\ell}(M(A))$
as functions of the $d_{\ell}(M(A))$.
\begin{lemma}\label{lem:d}
A collection of data $d_\ell(M(A))$ for each $A\subseteq E$ and $\ell\in\val(R)^+$
comes from a matroid over~$R$ if and only if the following hold:
\begin{enumerate}
\item[(L0)] for each $A\subseteq E$, the sequence $d_\ell(M(A))$ is bounded above and nonincreasing with~$\ell$;
\item[(L1)] for each $A\subseteq E$, $b\in E\setminus A$, and $\ell\in\val(R)^+$,
\[1 \geq d_\ell(M(A)) - d_\ell(M(Ab)) \geq 0;\]
\item[(L2a)]
for each $A\subseteq E$, $b,c\in E\setminus A$, and $\ell\in\val(R)^+$,
\[d_{\leq\ell}(M(A)) - d_{\leq\ell}(M(Ab)) - d_{\leq\ell}(M(Ac)) + d_{\leq\ell}(M(Abc)) \geq 0;\]
\item[(L2b)] equality holds in {\rm (L2a)} when $d_\ell(M(Ab))\neq d_\ell(M(Ac))$.
\end{enumerate}
\end{lemma}
In order to get on to the good stuff, we have placed
the proofs of Lemma~\ref{lem:d} and Proposition~\ref{prop:t}
in Appendix~\ref{app:d}.

\section{Three-term Pl\"ucker relations}\label{sec:D4}\label{sec3}
This section is dedicated to explaining
how the axioms for matroids over valuation rings
can be profitably understood in a unified fashion.
The unifying perspective works by imagining two kinds of extension of the
ground set, namely adding loops and adding ``generic'' elements
extending the matroid in the freest way.
Any sentence which is true of subsets
of ground set elements remain true when some of these
are specialised to generic elements or loops.
In particular, starting with a single axiom and making such specialisations
produces an entire system of axioms.
(Observe however that these specialisations are not \emph{implications}
unless one already knows that adding loops and making free extensions
produce matroids over~$R$, and to prove this requires the axioms.
So the axiom system cannot be reduced to only one member.)

Notionally, a ``generic'' element $g$ in a matroid $M$ over~$R$ should be an element
such that, for any subset $A$ of its ground set, $M(Ag)$ is the quotient
of $M(A)$ by a cyclic submodule in the generic way.
The next lemma shows that
this genericity is sensible from the perspective of commutative algebra.

\begin{lemma}\label{lem:nongeneric submodule}
Let $N$ be a finitely presented $R$-module, and let $C$ be the largest direct summand of~$N$.
The set of $n\in N$ which generate a cyclic submodule
which is not a summand of~$N$ isomorphic to~$C$ is a proper submodule $N_{\rm gen}$ of~$N$, namely
\[\m N + \{n\in N : rn=0\mbox{ for some $r\not\in\Ann(N)$}\}.\]
\end{lemma}

\begin{proof}
The number of cyclic summands of a module $N$ is $\dim_{R/\m} N/\m N$.
If $n\in\m N$ then this quantity fails to decrease in~$N/n$,
so $N/n$ is not the quotient of $N$ by a summand.
The annihilator of $N$ is the annihilator of its largest cyclic summand $C$;
so if $n$ is annihilated by a larger ideal of~$R$, it cannot generate
a copy of~$C$.  Conversely, if $n$ is in neither of these submodules, then
$\langle n\rangle\cong C\hookrightarrow N$, and the inclusion splits
by lifting a splitting $N/\m N\to \langle n\rangle/\m\langle n\rangle \cong R/\m$.
The submodule is proper because it clearly does not contain any generator
of a copy of~$C$.
\end{proof}

We define $N_{\rm gq}$ to be the complementary direct summand
to $C$ within~$N$.
Unless $N=0$, the submodule $N_{\rm gq}$ is not the whole of~$N$, since given
any choice of a direct summand $C$, its generator lies outside $N_{\rm gq}$;
in these cases $N_{\rm gq}$ is a proper $R/\m$-vector subspace of~$N$.
We may justifiably say that $N_{\rm gq}$ is the quotient of~$N$ by a \emph{generic} element;
the notation ``gq'' stands for ``generic quotient''.
In line with the fact that the data of matroids over rings contains
only isomorphism types of modules and no morphisms,
we will allow ourselves to say that any module isomorphic to~$N_{\rm gq}$ is a quotient by a generic element
of any module isomorphic to~$N$.

It is meaningful to ``add a generic element'' to a matroid over~$R$.

\begin{proposition}\label{prop:add-generic}
Let $M$ be a matroid over~$R$ on ground set $E$.  Define a collection of
$R$-modules $M_{\rm ge}$ indexed by subsets of $E\amalg\{g\}$ by
\[M_{\rm ge}(A) = \begin{cases}
M(A) & g\not\in A \\
M(A)_{\rm gq} & g\in A
\end{cases}\]
Then $M_{\rm ge}$ is a matroid over~$R$.
\end{proposition}

We say that $M_{\rm ge}$ is a \emph{generic extension} of~$M$.
The opposite notion is easier: let $L$ be the matroid over~$R$ with a single
element $0$ that is (globally) a loop, i.e.\ $L(\emptyset) = L(\{0\}) = 0$.
Then $M\oplus L$ extends $M$ by adding a loop.

\begin{proof}
Axiom~(M) for $M$ provides, for any $A\subseteq E$ and any $b,c\not\in A$,
elements $x,y\in M(A)$ as in Definition~\ref{def:matroid}.
To check the axiom for $M_{\rm ge}$,
it is enough to show that there is a single element $z\in M(A)$
such that the quotient of any of the modules
$M(A)$, $M(A)/\langle x\rangle$, $M(A)/\langle y\rangle$, $M(A)/\langle x,y\rangle$
by the image of~$z$ is a generic quotient,
as then a suitable choice of $x$, $y$, and~$z$ within some quotient of~$M(A)$
gives any of the squares called for by axiom~(M).

As noted, we may assume $R$ is any valuation ring with its value group
without consequence for which matroids over~$R$ exist.
So suppose $R$ has infinite residue field $\mathbf k$.
By Lemma~\ref{lem:nongeneric submodule}, the choices of~$z$
which yield a nongeneric quotient form a proper $\mathbf k$-vector subspace of~$M(A)$.
We similarly obtain proper vector subspaces of
$M(A)/\langle x\rangle$, $M(A)/\langle y\rangle$, and~$M(A)/\langle x,y\rangle$
giving nongeneric quotients, which remain proper when lifted to~$M(A)$.
The union of finitely many proper subspaces of a vector space
over an infinite field is still a proper subset;
any element of its complement will serve for~$z$.
\end{proof}

Observe that, if $N$ is a finitely generated $R$-module,
then $t_i(N_{\rm gq}) = t_{i+1}(N)$ for each $i\geq0$.
As such, if $M$ is a matroid over~$R$ and $A$ any subset of its ground set,
the data that would comprise the quotient by a ``generic'' element
$M(Ag)$ is actually found in the data of~$M$,
namely as the sequence $t_{\ast+1}(A;M)$.
With this in mind, we give the following definitions.

\begin{definition}
\begin{itemize}
\item The \emph{genericisation} of a sentence involving
the terms $t_i(A)$ for various sets $A$ and integers $i$,
along an element $b\in E$, is the sentence obtained
by replacing any term $t_i(Ab)$ whose set contains $b$ with the term $t_{i+1}(A)$.

\item The \emph{zeroisation} of a sentence involving
the terms $t_i(A)$ along an element $b\in E$
is the sentence obtained
by replacing any term $t_i(Ab)$ whose set contains $b$ with the term $t_i(A)$.
\end{itemize}
\end{definition}

Note that genericisation and zeroisation commute.
Given a true sentence $\mathcal S$, universally quantified over matroids $M$ over~$R$ and over
every set and element variable appearing therein, the genericisation
and zeroisation of~$\mathcal S$ along~$b$ are also true sentences.  Indeed, the genericisation
of~$\mathcal S$ asserts $\mathcal S$ of sets in $M_{\rm ge}$ containing the generic
element $g$ when the element $b$ is specified to be~$g$,
and the zeroisation asserts $\mathcal S$ of sets in $M\oplus L$
when $b$ is specified to be~0.

We now demonstrate that iterated genericisation and zeroisation allow us to obtain
several different axiom systems for matroids over a valuation ring starting from a
single statement.  The statement we choose as source is a
tropical Pl\"ucker relation.
The following relation holds between the Pl\"ucker coordinates $(p_{B} : B\subseteq[n],|B|=r)$
of any point in the Grassmannian $\Gr(r,\K^n)$:
$$p_{Abc}p_{Ade}-p_{Abd}p_{Ace}+p_{Abe}p_{Acd}=0.$$
By formally tropicalising this relation, using the symbol $t_i(B)$ for the valuation of
$p_B$, we get
\begin{equation}\tag{D4}
\min\big\{t_i(Abc)+t_i(Ade), t_i(Abd)+t_i(Ace), t_i(Abe)+t_i(Acd)\big\}\text{ is attained twice}.
\end{equation}
\begin{proposition}\label{prop:D4}
Let $R$ be a valuation ring.  Statement (D4) is true of all matroids over~$R$.
\end{proposition}
This follows by the proof
of Proposition~5.6 of~\cite{FM}, which goes through verbatim
using the quantities $t_i$ instead of $d_{\leq n}$,
and with Proposition~\ref{prop:t} providing the hypotheses.
In the case $i=0$, statement (D4) is the principal axiom of valuated matroids,
since the length of a module $N$ equals $t_0(N)$.
An alternative proof can be given making use of Lemma~\ref{lem:incidence implies pluecker}.

It follows directly from the above that statements obtained by repeatedly
genericising and zeroising (D4) are also true of all matroids over~$R$.
We give names to the nine statements obtained by genericising
and zeroising from zero to two times each:
the name (D$n^{(k)}$), where the $^{(k)}$ denotes $k$ primes,
is the statement obtained from (D4) by genericising $k$ times
and zeroising $n-k$ times.
In particular the digit $n$ gives the
number of elements of $E\setminus A$ which are quantified over,
following the pattern of our other axiom names.

The genericisation of (D4) along~$e$ is:
\begin{equation}\tag{D3$'$}
\min\big\{t_i(Abc)+t_{i+1}(Ad), t_i(Abd)+t_{i+1}(Ac), t_{i+1}(Ab)+t_i(Acd)\big\}\text{ is attained twice}
\end{equation}
and its zeroisation is:
\begin{equation}\tag{D3}
\min\big\{t_i(Abc)+t_i(Ad), t_i(Abd)+t_i(Ac), t_i(Ab)+t_i(Acd)\big\}\text{ is attained twice}.
\end{equation}
Applying genericisation to~$e$ and zeroisation to~$d$ yields
\begin{equation}\tag{D2$'$}
\min\big\{t_i(Abc)+t_{i+1}(A), t_{i+1}(Ab)+t_i(Ac), t_i(Ab)+t_{i+1}(Ac)\big\}\text{ is attained twice}.
\end{equation}
Generising or zeroising twice leads to a duplication of one of the terms
in the minimum, respectively
\begin{gather*}
\min\big\{t_i(Abc)+t_{i+2}(A), t_{i+1}(Ab)+t_{i+1}(Ac), t_{i+1}(Ab)+t_{i+1}(Ac)\big\}\text{ is attained twice},
\\
\min\big\{t_{i}(Abc)+t_i(A), t_{i}(Ab)+t_{i}(Ac), t_{i}(Ab)+t_{i}(Ac)\big\}\text{ is attained twice},
\end{gather*}
which simplify to
\begin{gather}\tag{D2$''$}
t_i(Abc)+t_{i+2}(A)\geq t_{i+1}(Ab)+t_{i+1}(Ac),
\\\tag{D2}
t_{i}(Abc)+t_i(A)\geq t_{i}(Ab)+t_{i}(Ac).
\end{gather}
By iterating we get:
\begin{gather}\tag{D1$''$}
t_i(Ab)+t_{i+2}(A)\geq t_{i+1}(Ab)+t_{i+1}(A),
\\\tag{D1$'$}
t_{i+1}(Ab)+t_i(A)\geq t_{i}(Ab)+t_{i+1}(A),
\\\tag{D0$''$}
t_i(A)+t_{i+2}(A)\geq 2 t_{i+1}(A).
\end{gather}

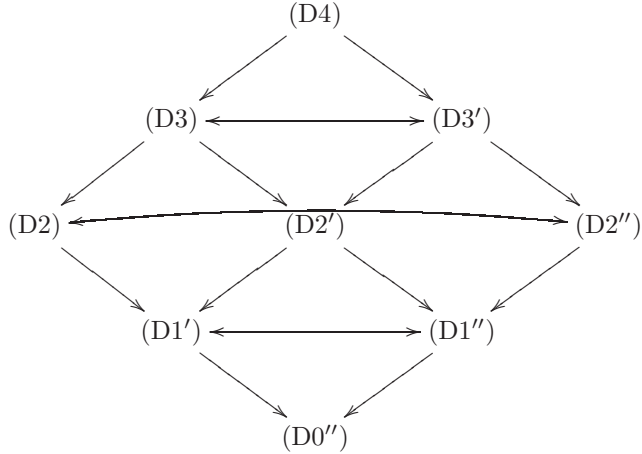
\begin{figure}[hbt]
$$
\xymatrix{
    &    &   (\mathrm D4)\ar[dl]\ar[dr]    &    &   \\
    &(\mathrm D3)\ar[dl]\ar[dr]\ar@{<->}[rr]  &   &  (\mathrm D3')\ar[dl]\ar[dr]  &   \\
(\mathrm D2)\ar[dr]\ar@{<->}@/^/[rrrr]  &    &(\mathrm D2')\ar[dl]\ar[dr]    &    &(\mathrm D2'')\ar[dl]\\
    &(\mathrm D1')\ar[dr]\ar@{<->}[rr]  &   &  (\mathrm D1'')\ar[dl]  &  \\
    &    &   (\mathrm D0'')    &    &
}
$$
\caption{A schematic of the relationships between the nine statements
of this section.}\label{fig:axioms}
\end{figure}
We can organise these axioms in a diagram, Figure~\ref{fig:axioms},
in which rightward descending arrows are genericisations, leftward descending arrows are zeroisations,
and horizontal double arrows are \emph{dualisations}, obtained by replacing
the quantities $t_i(A)$ in the fashion of matroid duality,
to wit
\begin{equation}\label{eq:dual}
t_i(M(A)) = t_{|A|-r+i}(M^*(E\setminus A))
\end{equation}
where $r$ is the generic rank of~$M$.
The statements (D4), (D2$'$), and (D0$''$) on the central axis are self-dual.

We remark that the relation (D0$''$) is axiom~(T0),
while (D1$''$) and (D1$'$) together are
axiom~(T1), and (D2$'$) is axiom~(T2).
Therefore, these four (D) statements, together with (TS), suffice to characterise matroids over~$R$.
Moreover, relation (D2) is submodularity of the $i$-fold generic quotient of~$M$,
and (D3) corresponds to the content of Proposition~5.5 of~\cite{FM}.

\begin{remark}\label{rem:arithmetic}
We compare quasi-arithmetic and arithmetic matroids \cite{D'Adderio-Moci, Branden-Moci}
to matroids over~$\mathbb Z$.
Recall that a \emph{quasi-arithmetic matroid} on finite ground set $E$ is formalised
as a (usual) matroid on the set~$E$ together with the extra data of a \emph{multiplicity function} $m: 2^E \to \mathbb N$,
subject to some axioms; from our point of view,
the matroid is $M\otimes_{\mathbb Z}\mathbb Q$ while $m(A)$ is the cardinality of the torsion subgroup of~$M(A)$.
\emph{Arithmetic matroids} are quasi-arithmetic matroids
which further satisfy a positivity axiom (P)
which (in the realisable case) arises from topology of the associated toric arrangement.

These structures differ from matroids over~$\mathbb Z$ 
in a significant way which was not brought to light in~\cite{FM}.
Namely, no analogue of axiom (D4) or its genericisations and zeroisations
are true for (quasi-)arithmetic matroids, and so the connection to valuated matroids is absent.

For example, if an arithmetic matroid whose underlying matroid is uniform of rank~2 on the set $\{1,2,3,4\}$
has multiplicity function $m$ assigning 1 to every set of cardinality not equal to~2,
then the multiplicities of sets of size~2 can be arbitrary, independently chosen positive integers.
But a (quasi-)arithmetic matroid arising from a matroid over~$\mathbb Z$ must satisfy
(D4) for $i=0$ at every prime $p$, namely that
\[\min\{v_p(m(12))+v_p(m(34)), v_p(m(13))+v_p(m(24)), v_p(m(14))+v_p(m(23))\}\]
is attained twice, where $v_p:\mathbb Q\to\mathbb Z\cup\{\infty\}$ is the $p$-adic valuation.
\end{remark}

\section{Spannable matroids}\label{sec:spannable}
Spanning sets play an important role in classical matroid theory.
However, many perfectly reasonable matroids over~$R$ have no
sets $A\subseteq E$ such that $M(A)=0$, the natural analogue of spanning sets.
The focus of this short section is those matroids that do have such sets~$A$.
They enjoy extra structural properties compared to arbitrary matroids over~$R$.
We will make use of these structural properties in Section~\ref{sec:parameter}.

\begin{definition}
Let $M$ be a matroid over~$R$.  We say that $M$ is \emph{spannable} if $M(E)=0$.
\end{definition}

\begin{lemma}\label{lem:rk 2 generic}
Let $M$ be a matroid over~$R$ on ground set $\{1,2\}$.
If neither $M(\{1\})$ nor $M(\{2\})$ is a quotient of $M(\emptyset)$ by a generic element,
then $M(\{1,2\})$ is a quotient of neither $M(\{1\})$ nor $M(\{2\})$ by a generic element.
\end{lemma}

\begin{proof}
Choose elements $x$ and $y$ in $M(\emptyset)=:N$ permitting the identifications
$M(\{1\})=N/\langle x\rangle$,
$M(\{2\})=N/\langle y\rangle$, and
$M(\{1,2\})=N/\langle x,y\rangle$.
By Lemma~\ref{lem:nongeneric submodule}, both $x$ and~$y$ lie in
the submodule $N_{\rm gen}$ defined there.
The preimage of $(N/\langle x\rangle)_{\rm gen}$ in $N$
contains $N_{\rm gen}$: this is easy to check keeping
in mind that $\Ann(N/\langle x\rangle) = \Ann(N)$, both
of these being equal to the annihilator of the largest cyclic summand
of $N$, which is preserved in $N/\langle x\rangle$.
Therefore the image of $y$ in $N/\langle x\rangle$
falls in $(N/\langle x\rangle)_{\rm gen}$; this and its analogue with
$x$ and $y$ reversed are the lemma.
\end{proof}

\begin{lemma}\label{lem:spannable generic}
If $M$ is a spannable matroid over~$R$ on ground set~$E$, then for every
proper subset $A\subsetneq E$ there is an element $b\in E\setminus A$
such that $M(Ab)$ is a quotient of $M(A)$ by a generic element.
\end{lemma}

\begin{proof}
The result is true when $|A| = |E|-1$, as then $M(A)$ is cyclic and $M(E)=0$.
So if $A$ were a maximal set for which the conclusion
of the lemma didn't hold, we would have $|A|\leq|E|-2$.
Let $b\in E\setminus A$.  For every element $c$ of $E\setminus(Ab)$,
Lemma~\ref{lem:rk 2 generic} shows that $M(Abc)$ is not a
quotient of $M(Ab)$ by a generic element, and therefore the
conclusion is also false of $Ab$, contradiction.
\end{proof}

\begin{corollary}\label{cor:sg1}
If $M$ is a spannable matroid over~$R$, then
\begin{equation}\label{eq:sg1}
t_{i+1}(M(A)) = \min_{b\not\in A}\{t_i(M(Ab))\}
\end{equation}
for all $i\geq0$ and $b\not\in A$.
\end{corollary}

\begin{proof}
Section~\ref{sec:D4} tells us that if $N_{\rm gq}$ is a generic quotient of~$N$,
then $t_i(N) = t_{i+1}(N_{\rm gq})$.  This implies the corollary
given the observation that the quantities $t_i(N/\langle x\rangle)$
are all simultaneously minimised when $N/\langle x\rangle\cong N_{\rm gq}$.
This observation is equivalent, by the techniques of Proposition~\ref{prop:t},
to the fact that the same quotients
simultaneously maximise $d_\ell(N/\langle x\rangle)$ for all $\ell\in\val(R)^+$.
Indeed, $d_\ell(N/\langle x\rangle)$ equals either
$d_\ell(N)$ or $d_\ell(N)-1$, and equals the former
when $\ell$ exceeds the valuation of an element annihilating~$N$.
And if $N/\langle x\rangle\cong N_{\rm gq}$, then
$d_\ell(N/\langle x\rangle) = d_\ell(N)-1$ for all $\ell$ not this large.
\end{proof}

\begin{corollary}\label{cor:spannable determined}
A spannable matroid $M$ over~$R$ is determined by the data
$t_0(A) = \length(M(A))$ for all $A\subseteq E$.
\end{corollary}

\begin{proof}
The values of $t_i(A)$ for $i>0$ are obtained by $i$-fold application of Corollary~\ref{cor:sg1}.
\end{proof}

\begin{corollary}\label{cor:i>=l}
Let $M$ be a matroid over~$R$ on ground set $E$,
and suppose $M(E)$ can be presented with $\ell$ generators.
Then \eqref{eq:sg1} holds when $i\geq\ell$,
and therefore $M$ is determined by the data
$t_i(A)$ for all $A\subseteq E$ and $0\leq i\leq\ell$.
\end{corollary}

\begin{proof}
Apply Corollary~\ref{cor:spannable determined} to the $\ell$-fold
generic extension $M'$ of~$M$ (Proposition~\ref{prop:add-generic}),
which is spannable and of which $M$ is a deletion.
If $G$ is the set of $\ell$ generic elements added, then
\begin{align*}
t_{j+\ell+1}(M(A)) &= t_{j+1}(M'(A\cup G))
\\&= \min_{b\not\in A}\{t_j(M'(Ab\cup G))\}
\\&= \min_{b\not\in A}\{t_{j+\ell}(M(Ab))\}
\end{align*}
for any $j\geq0$.
\end{proof}

The dual of a matroid $M$ over~$R$ is spannable if and only if $M(\emptyset)$ is a free module.
It will be convenient to have terminology for these as well.
\begin{definition}
Let $M$ be a matroid over~$R$.  We say that $M$ is \emph{cospannable} if its dual is spannable,
and \emph{bispannable} if it is both spannable and cospannable.
\end{definition}

\begin{corollary}\label{cor:span cospan}
A bispannable matroid $M$ over~$R$ is determined by the modules $M(A)$ with $|A|=r$,
where $M(\emptyset)\cong R^r$ (i.e.\ $r$ is the generic rank of~$M$).
\end{corollary}

\begin{proof}
For $|A|<r$, Lemma~\ref{cor:sg1} determines the values of $t_i(A)$ when $i+|A|\geq r$,
while $t_i(A)$ must equal $\infty$ when $i+|A|<r$.
The same argumentation on the dual recovers the remaining data.
\end{proof}

We expect that there are other interesting cryptomorphisms to be found for the spannable matroids.

\begin{question}
Let $G$ be an ordered cancellative abelian monoid.
How can one characterise the functions assigning to each subset $A\subseteq E$
an element $t_0(A)\in G$ that arise from some spannable matroid over a valuation ring?
\end{question}

\begin{question}\label{q:hyperfield}
Is there a hyperring structure $H$ on the set of $R$-modules such that
bispannable matroids over~$R$ are exactly matroids over the hyperring~$H$?
\end{question}

It may also be possible to write down a meaningful polyhedron based on
the rank description of the usual matroid polytope,
unlike the polyhedron $P(M)$ of Section~\ref{sec:P} that is our focus in this paper,
which takes off from the vertex description.

\begin{question}
Let $M$ be a spannable matroid over~$R$.  Define the polytopal cone
\begin{align*}
RP(M) = \{x\in\R^{E\times\val(R)^+} :\quad  & x_{a,\ell}\geq x_{a,m}\mbox{ for all $\ell<m\in\val(R)^+$};
\\& \sum_{a\not\in A}x_{a,\ell}\geq d_{\ell}(M(A))\mbox{ for all $\ell\in\val(R)^+$, $A\subseteq E$}\}.
\end{align*}
Characterise the polyhedra that arise as $RP(M)$.
Does $RP(M)$ determine $M$?  What are its vertices?
\end{question}

\section{The polyhedron $P(M)$}\label{sec:P}

In this section we introduce an analogue of the
matroid polytope for matroids over~$R$,
where $R$ is a valuation ring of height one.  The height one assumption
is what guarantees that the polyhedra we produce will be subsets of real vector spaces.
We first recall the definition of valuated matroid,
and then the definition and cryptomorphic nature of the
\emph{matroid (basis) polytope} for usual matroids,
introduced in~\cite{Edmonds,GGMS}, and the analogous polyhedron
associated to a valuated matroid, treated by Speyer in \cite[Prop~2.2]{Speyer}.

Throughout this section, the ground set $E$ of all matroids and matroids over rings
will have $|E| = n$.  We will identify $\R^E$ and its basis $\{e_a : a\in E\}$
with $\R^n$ and its standard basis, thereby implicitly identifying
$E$ with~$[n]$.  Define $e_A = \sum_{a\in A}e_a$.

``Valuated matroid'' was the name given by Dress and Wenzel~\cite{DW}
to their generalisation of matroids to encode the valuation data associated to
linear dependences of vectors in a field with nonarchimedean valuation.
A \emph{valuated matroid} on ground set $E$ of rank $r$
is a function~$v:\binom Er\to\R\cup\{\infty\}$
satisfying equation (D4) with $t_i=v$, that is
\begin{equation*}
\min\big\{v(Abc)+v(Ade), v(Abd)+v(Ace), v(Abe)+v(Acd)\big\}\text{ is attained twice},
\end{equation*}
for all $A\in\binom E{r-2}$ and $b,c,d,e\not\in A$,
and such that the support $v^{-1}(\R)$ of~$v$ is the set of bases of a matroid.

In tropical geometry valuated matroids play a central role as the data determining a linear space.
In this aspect they are given the name \emph{tropical Pl\"ucker vectors}: see further discussion in Section~\ref{sec:TLS}.
Speyer introduced the study of tropical linear spaces.
His work identified them as certain subcomplexes of the normal complex to a regular subdivision of a matroid polytope.
In the present work we focus not on the regular subdivision but on the polyhedron in one more dimension inducing it,
which we refer to as the \emph{lifted} polyhedron.

If $M$ is a matroid on ground set $E$, its basis polytope is defined to be
\[P(M) = \conv\{e_B : \mbox{$B$ a basis of $M$}\}\subseteq\R^E.\]
If $(M,v)$ is a rank~$r$ valuated matroid on ground set $E$,
its lifted polyhedron is
\[P(M) = \conv\{(e_A,v(A)):A\in{\textstyle\binom Er}\}+\cone\{(\underline 0,1)\}\subseteq\R^E\times\R,\]
with the understanding that points $(e_B,v(B))$ with an infinite coordinate
$v(B)=\infty$ should be omitted from the convex hull (from the projective point of view,
they are made redundant by the Minkowski sum with the ray).
Note that all of the points $e_B$, respectively $(e_A,v(A))$,
are vertices of~$P(M)$.

Let $\pi:\R^E\times\R\to\R^E$ be the projection onto the first factor.
Recall the definition of recession cone from Section~\ref{ssec:conventions}.
\begin{theorem}\label{thm:matroid polytope}\mbox{}
\begin{enumerate}\renewcommand{\labelenumi}{(\alph{enumi})}
\item \cite[Thm.~4.1]{GGMS} If $M$ is a matroid, then each edge of~$P(M)$
has direction vector $e_i-e_j$ for some $i,j\in E$.

Conversely, any polytope $P\subseteq\R^E$ all of whose
vertices lie in $\{0,1\}^E$ and all of whose edges have direction
vectors of form $e_i-e_j$, occurs as $P(M)$ for some~$M$.
\item \cite[Prop.~2.2]{Speyer} If $(M,v)$ is a valuated matroid, then the projection of
each edge of~$P(M)$ by~$\pi$
has direction vector $e_i-e_j$ for some $i,j\in E$.

Conversely, any polyhedron $P\subseteq\R^E\times\R$ whose recession cone is
$\cone\{(\underline 0,1)\}$, the projections of whose vertices by~$\pi$
lie in $\{0,1\}^E$, and the projections of whose edges by~$\pi$
all have direction vectors of form $e_i-e_j$, occurs as $P(M)$ for some~$M$.
\end{enumerate}
\end{theorem}
The forward implication of part~(a) was obtained decades earlier by Edmonds~\cite{Edmonds},
of which it follows from Proposition~19 and the forms of the inequalities in~(17).
In part~(b) we allow $i=j$.

\subsection{Definition and characterisation}

Now let $M$ be a matroid over~$R$ on ground set~$E$.
\begin{definition}
Define the polyhedron
\[P(M) = \conv\{(e_A,i,t_i(A)) : A\subseteq E, i\in\Z_{\geq0}\}
+ \cone\{(\underline 0,0,1)\}\subseteq\R^E\times\R^2.\]
\end{definition}
Again, points $(e_A,i,t_i(A))$ where $t_i(A)=\infty$ are to be excluded
from the convex hull.  Note that $P(M)$ has a finite number of vertices, because for
each $A\subseteq E$, the sequence of $t_i(A)$ eventually becomes 0,
say for $i\geq i_0$; then the point $(e_A,i,0)$ is not a vertex
as it is a convex combination of $(e_A,i_0,0)$ and $(e_A,i+1,0)$.
Similarly we observe that the recession cone of~$P(M)$ is
$$\rec P(M)= \cone\{(\underline 0,1,0),(\underline 0,0,1)\}.$$

The ring $R$ being fixed, $M$ is determined by $P(M)$, because
\[t_i(A) = \min\{y\in\R:(e_A,i,y)\in P(M)\}.\]

We now present a cryptomorphic axiom system for matroids over a fixed valuation ring $R$.
The main difference from the valuated matroid setting is that the set of
allowable positions of vertices is enlarged by two dimensions,
and the set of edge directions is enriched to match.
Valuated matroid polyhedron edges
which do not contract to a point on projection lie in the direction
of one of the roots $e_i-e_j$, $i\neq j$ of the type $A_{n-1}$ root system.
Matroids over~$R$ replace this with a type $A_{n+1}$ root system, with the extra
two dimensions corresponding to genericising and zeroising.
Let $\pi:\R^E\times\R\times\R\to\R^E\times\R$ be the projection onto the first two factors.

\begin{theorem}\label{thm:polyhedron}
The projection of each edge of $P(M)$ by~$\pi$
has direction vector of form $e-e'$, for $e$ and~$e'$ elements of the set
$D = \{(e_a,0) : a\in E\}\cup\{(\underline 0,0),(\underline 0,1)\}$.

Conversely, let $P\subseteq\R^E\times\R^2$ be a polyhedron
satisfying the following conditions:
\begin{enumerate}\renewcommand{\theenumi}{\roman{enumi}}
\item the recession cone of $P$ is
\[\cone\{(\underline 0,1,0),(\underline 0,0,1)\};\]
\item $P$ contains $[0,1]^E\times[i,\infty)\times[0,\infty)$ for some~$i$;
\item all vertices of $P$ lie in $\{0,1\}^E\times\N\times\val(R)$;
\item for any $e_A\in\{0,1\}^E$ and $i\in \N$,
we have $\{\min\{y\in\R:(e_A,i,y)\in P\}\}\in\val(R)$;
and
\item each edge of $P$ has a direction vector whose image under $\pi$ is $e-e'$ for some $e,e'\in D$.
\end{enumerate}
Then $P$ arises as $P(M)$ for some matroid $M$ over~$R$.
\end{theorem}

Although the statement of Theorem~\ref{thm:polyhedron} reads rather more complex
than that of Theorem~\ref{thm:matroid polytope}, we suggest that it is similarly
simple once one takes account of the more complicated structure of the data which
needs to be recorded.  Condition~(v) is the meaty one.
The new edge directions allowed by the inclusion
of $(\underline 0,0)$ and~$(\underline 0,1)$ in the set~$D$
correspond to zeroisation and genericisation,
but they are also not foreign to classical matroid theory:
for instance, edges of the former sort, in directions $e_i$,
appear in independent set polytopes \cite[Chapters~39--40]{Schrijver}.
Indeed, as we'll see in the next subsection, independent set polytopes
appear as faces of~$P(M)$; this fact together with closure under duality
can also be taken to account for the new directions.

The proof calls on a lemma that can be seen as a generalisation of symmetric basis exchange.
Momentarily write $t((e_A,i))$ to mean $t_i(A)$;
the argument of this $t$ is a vector in $\mathbb R^{|E|+1}$.

\begin{lemma}\label{lem:gen exch}
Let $v$ and $w$ be vectors in~$\{0,1\}^n\times\N$.
Write $v_{n+2}$ to denote $-\sum_{i=1}^{n+1} v_i$, and similarly for~$w$.
Let $a\in[n+2]$ be such that $v_a>w_a$,
and $B$ be the set of coordinates $b\in[n+2]$
such that $v_b<w_b$.
Then $t(v)+t(w)$ is not strictly less than every sum of the form
\[t(v-e_a+e_b)+t(w+e_a-e_b)\]
for $b\in B$.
\end{lemma}

Observe that some choice of~$a$ will exist unless $v=w$, and in this case $B$ is nonempty.

\begin{proof}
This lemma again follows by the same argument
used to prove Proposition~5.6 of~\cite{FM},
with the quantities $d_{\leq n}$ formally replaced by $t_i$.
That proposition is framed as a tropical Pl\"ucker relation.
The statement of the lemma is recovered by
letting the ground set $E$ be partitioned into $n+2$ parts,
\[E = E_1\mathbin{\dot\cup}E_2\mathbin{\dot\cup}\cdots\mathbin{\dot\cup} E_{n+2},\]
and replacing $v$ by a set containing $v_i+c_i$ elements of $E_i$ for each $i\in[n+2]$
and similarly for $w$,
where $c_i$ are constants chosen to make these numbers nonnegative.

The equivalents of the two hypotheses about the quantities~$p$
used by the proof of Proposition~5.6 of~\cite{FM} are provided by (D3) and~(D4).
\end{proof}

\begin{proof}[Proof of Theorem~\ref{thm:polyhedron}]
Let $F$ be an edge of $\widetilde P(M)$ not in direction $(\underline 0,0,1)$,
and $f$ a linear functional minimised on~$F$.
Since $P(M)$ contains rays in direction $(\underline 0,0,1)$, the $(n+2)$th coefficient of~$f$ is positive.
Let $(v,t(v))$ and $(w,t(w))$ be two points of $F$
so that $v$ and $w$ are two lattice points appearing consecutively on~$\pi(F)$.
Suppose $v-w$ is not of the form $e-e'$, for $e$ and $e'$ elements of~$D$.
Then $v$ and~$w$ are unequal, so they satisfy the hypotheses of Lemma \ref{lem:gen exch}.
The lemma yields another pair
of lattice points $v'=v-e_a+e_b$ and $w'=w+e_a-e_b$ such that
\[t(v)+t(w)\geq t(v')+t(w')\]
and, by construction, $v'+w'=v+w$.
It follows that
\[f\left(\frac{v+w}2,\frac{t(v)+t(w)}2\right) \geq f\left(\frac{v'+w'}2,\frac{t(v')+t(w')}2\right).\]
But $((v+w)/2,(t(v)+t(w))/2)$ is a point of~$F$, the set of minima of~$f$,
and hence $((v'+w')/2,(t(v')+t(w'))/2)$ is as well.
Since these agree in all coordinates but their last, and $F$ is not in direction $(\underline 0,0,1)$,
they are the same point, implying $t(v)+t(w) = t(v')+t(w')$.
Thus $f$ takes its minimum value on $(v',t(v'))$ and $(w',t(w'))$ as well,
and these points are also on~$F$.
This is a contradiction, as $F$ is an edge and these points are not collinear with $v$ and~$w$.

Turning to the converse,
we take the values $t_i(A)$ to be given by
\[t_i(A) = \{\min\{y\in\R:(e_A,i,y)\in P\}\},\]
so that $t_i(A)\in\val(R)$ by~(iv).
It suffices to show the axioms of Proposition~\ref{prop:t} for these,
as then the $t_i(A)$ will define a matroid $M$ over~$R$,
and then $P$ will agree with $P(M)$ by (i) and~(iii).
Axiom (TS) is an immediate translation of assumption~(ii),
while axiom (T0) follows from the convexity of~$P$.

To prove axiom~(T1), we fix $A$ and~$b$ and restrict attention to
the 3-dimensional face $F$ of~$P$ where all coordinates $x_i$ for $i\neq b$ are fixed
to the value 0 (if $i\in A$) or 1 (if $i\not\in A$).
The vertices of~$F$ are all of the form
\[v_{S,i} = (e_S,i,t_i(S))\]
where the set $S$ equals $A$ or $Ab$, and $i\in\N$.
All such points $v_{S,i}$ are contained in~$F$, though they needn't all be vertices.
Let $\Sigma$ be the regular subdivision of the point set
$\{e_A,e_{Ab}\}\times\N$ induced by $F$.

We claim that, for all $i\in\N$, the points $v_{A,i}$, $v_{Ab,i}$, and~$v_{A,i+1}$
are all contained in a common proper face of~$F$,
as are the points $v_{Ab,i}$, $v_{A,i+1}$, and~$v_{Ab,i+1}$.
It is enough to prove that the images under $\pi$ of these points each lie on a common face of~$\Sigma$.
For concreteness let's look at the case of $\pi(v_{A,i})$, $\pi(v_{Ab,i})$, and~$\pi(v_{A,i+1})$;
the other will be similar (the two are exchanged by duality).
If no face of~$\Sigma$ contains these three points,
then there must be either a vertex of~$\Sigma$ in
the interior of their convex hull~$C$---there is not---%
or an edge $e$ of~$\Sigma$ cutting across~$C$.
Not both endpoints of~$e$ are of the form $\pi(v_{S,j})$ for $j\leq i$,
since the convex hull of all these points is disjoint from the interior of~$C$.
Similarly, not both endpoints of~$e$ are of form $\pi(v_{S,j})$ with $j+(|S|-|A|)\geq i+1$;
not both have $S=A$; and not both have $S=Ab$.
All of these constraints together require $e$ to be in a direction not allowed
by condition~(v).

This proves the left inequality of statement~(T1) (the right inequality is the dual case).
For if $v_{A,i}$, $v_{Ab,i}$, and~$v_{A,i+1}$ are all contained in a proper face of~$F$,
therefore of $P$, there is some supporting hyperplane $H$ of~$P$ containing all three.
The point $w=v_{Ab,i}+v_{A,i+1}-v_{A,i}$ lies on~$H$.
By~(i), $P$ lies on the side of~$H$ where the last coordinate increases.
So since $v_{Ab,i+1}$ agrees in all but the last coordinate with~$w$,
its last coordinate must be at least as great,
\[t_{i+1}(Ab)\geq t_i(Ab)+t_{i+1}(A)-t_i(A).\]

Our proof of axiom~(T2) also runs along the above lines but is more tedious.
The relevant face of~$P$ now is the 4-dimensional face $F$ whereon
$x_i=0$ for $i\not\in Abc$ and $x_i=1$ for $i\in A$.
This projects under~$\pi$ to a 3-dimensional regular subdivision $\Sigma$.
There are a number of triples of points of~$F$
which must lie on common faces.  Beyond those argued as part of the last case,
this includes the triples of forms
$(v_{A,i}, v_{Ab,i}, v_{Ac,i})$,
$(v_{Ab,i}, v_{Ac,i}, v_{Abc,i})$,
$(v_{A,i+1}, v_{Ab,i}, v_{Ac,i})$, and
$(v_{Ab,i+1}, v_{Ac,i+1}, v_{Abc,i})$.
If one of the above triples $T$ did not lie on a common face, then
there is a 2-face $G$ which meets their convex hull in its interior.
This face $G$ meets the affine hull of~$T$ in a segment $s$;
it can be argued as before that this segment must be
in one of the directions not permitted by condition~(v).
So $G$ cannot have an edge in direction~$s$.
It must therefore have at least two edges in directions not parallel to
the affine hull of~$T$ whose span includes the direction of~$s$;
but condition~(v) in fact rules this out as well.

Consider the octahedron $O$ with vertices
$\pi(v_{Ab,i})$, $\pi(v_{Ac,i})$, $\pi(v_{Abc,i})$,
$\pi(v_{A,i+1})$, $\pi(v_{Ab,i+1})$, $\pi(v_{Ac,i+1})$.
Each of the faces of $O$ is contained in a face of~$\Sigma$,
so the intersection of~$\Sigma$ with~$O$ is a regular subdivision of~$O$.
Now, the centre point of~$O$ can be written as an affine combination
of its vertices in exactly three extremal ways, namely
as the average of pairs of opposite vertices.
If exactly one of these pairs has minimum average last coordinate in~$F$,
then the edge spanned by the two vertices appears in $\Sigma$,
and is in a direction not permitted by~(v).
So the minimum average last coordinate must be attained by two or more of these pairs,
and this is the content of~(T2).
\end{proof}

\subsection{The global case}

It is a natural step to generalise our cryptomorphism to the case of a unique factorisation domain
by introducing a new coordinate for each prime.
One way to do this is the subject of Proposition~\ref{prop:Dedekind poly}.
But Example~\ref{ex:two primes} shows the resulting polyhedra
cannot be characterised in terms of edges, as they can for the valuation ring case.

If $R$ is a UFD and $M$ is a matroid over~$R$,
let $t_i(A)\in(\N\cup\{\infty\})^{\maxSpec R}$ be the vector indexed by the maximal primes $p\in\maxSpec R$
whose $p$\/th component is $t_i(M_p(A))$, and define
\begin{multline*}
P(M) = \conv\{(e_A,i,t_i(A)) : A\subseteq E, i\in\Z_{\geq0}\}
\\+ \cone\{(\underline 0,1,0),(\underline 0,0,e_p):p\in\maxSpec R\}
\\ \subseteq\R^E\times\R\times\R^{\maxSpec R},
\end{multline*}
where $e_p$ is the standard basis vector for the coordinate corresponding to~$p$,
and again vertices which have an infinite coordinate are thrown out.
For each $p\in\maxSpec R$, let
\[\pi_p:\R^E\times\R\times\R^{\maxSpec R}\to\R^E\times\R\times\R\]
be the projection retaining only the coordinate corresponding to~$p$ from the last of the three blocks of coordinates.
It is not concerning that $P(M)$ is of infinite dimension,
because for all but finitely many primes $p\in\maxSpec R$ 
all modules $M(A)\otimes_R R_p$ are free, so that
$\pi_p(P(M))$ is the Cartesian product of a polyhedron in~$\R^E\times\R$ with~$[0,\infty)$.
Thus there is a linear projection of~$P(M)$ onto a finite dimensional subspace 
(discarding these coordinates $p$)
such that the restriction to the polytope is an affine isomorphism.

\begin{proposition}\label{prop:Dedekind poly}
Let $R$ be a UFD.
Let $P\subseteq\R^E\times\R\times\R^{\maxSpec R}$ be a polyhedron satisfying the following conditions:
\begin{enumerate}\renewcommand{\theenumi}{\roman{enumi}}
\item the recession cone of $P$ is
\[\cone\{(\underline 0,1,0),(\underline 0,0,e_p):p\in\maxSpec R\};\]
\item $P$ contains $[0,1]^E\times[i,\infty)\times[0,\infty)^{\maxSpec R}$ for some~$i$;
\item all vertices of $P$ lie in $\{0,1\}^E\times\N\times\N^{\maxSpec R}$; and
\item for any $e_A\in\{0,1\}^E$ and $i\in \N$,
the set $\{y\in\R:(e_A,i,y)\in P\}$ is either empty or an orthant in the $\R^{\maxSpec R}$ directions
based at a point of $\N^{\maxSpec R}$.
\end{enumerate}
Then $P$ arises as $P(M)$ for some matroid $M$ over~$R$
if and only if, for every prime $p\in\maxSpec R$,
the image $\pi_p(P)$ satisfies the edge direction condition of Theorem~\ref{thm:polyhedron}.
Moreover, any polytope $P(M)$ of a matroid $M$ over~$R$ satisfies (i) through~(iv) above,
and if $P(M)=P(M')$ for a matroid $M'$ over~$R$, then $M=M'$.
\end{proposition}

\begin{proof}
By \cite[Prop.~6.2]{FM},
a matroid $M$ over~$R$ consists of
a collection of matroids $M_p$ over the localisations $R_p$ for each $p\in\maxSpec R$
so that the classical matroids $M_p\otimes_{R_p}\Frac(R)$ all agree.
This says that given~$M$,
for each $i$ and~$A$, the vector $t_i(A)$ has either all finite or all infinite entries.
Therefore $\pi_p(P(M)) = P(M_p)$,
as we are not in the situation where some vertex of $\pi_p(P)$ which should be present is discarded
because $t_i(M_q(A))=\infty$ for some prime $q\neq p$.
The remainder of the content of the preposition is
an elementary translation of conditions (i) through~(iv) of Theorem~\ref{thm:polyhedron} to the present setting.
\end{proof}

If $R$ is a general Dedekind domain,
then \cite[Prop.~6.2]{FM} imposes conditions on the determinants of the modules $M(A)$,
an invariant valued in the Picard group of~$R$.
For a UFD the Picard group is trivial, so this is of no concern.
But when the Picard group is nontrivial, we would have to impose a simultaneous condition
on all the vertices of~$P$ with first block of coordinates $e_A$.
We leave investigation of this case to future work.

\begin{example}\label{ex:two primes}
The polyhedron $P(M)$ of Proposition~\ref{prop:Dedekind poly}
can have edges in non-matroidal directions.
Let $M$ be the realisable matroid on $\{1,2,3,4\}$ over $\mathbb Z$
with $M(\emptyset)=\Z^2$, realised by the vector configuration
$v_1 = (0,1)$,
$v_2 = (1,3)$,
$v_3 = (1,2)$,
$v_4 = (1,0)$.
Two of the vertices of $P(M)$ are $((1,0,0,1),0,(0,0,\ldots))$
and $((0,1,1,0),0,(0,0,\ldots))$, since the corresponding quotients
$\Z^2/\langle v_1,v_4\rangle$ and $\Z^2/\langle v_2,v_3\rangle$ are trivial modules.
The reader may confirm that there is an edge between these two vertices:
for instance, if the coordinates be labelled
$((x_1,x_2,x_3,x_4),i,(t_2,t_3,\ldots))$,
then the linear functional $3x_1+2x_2+2x_3+x_4+4(i+t_2+t_3)$
attains its minimum value $4$ only at the two aforementioned vertices.
The direction of this edge is disallowed by Theorem~\ref{thm:polyhedron}.

This example (like any example) can be brought into finite dimension.  
Choosing the base ring to be not $\Z$ but its localisation $(\Z\setminus(2\Z\cap3\Z))^{-1}\Z$,
whose only maximal primes are $\langle 2\rangle$ and $\langle 3\rangle$,
makes $P(M)$ have dimension $4+1+2$.
\end{example}

\subsection{Operations on matroids}
The polytope $P(M)$ transforms well under standard matroid operations on~$M$.
For the purposes of the next proposition, let $\{x_a : a\in E\}$
be the standard coordinates on $\R^E$.
Recall that the \emph{generic rank} of~$M$ is the rank of the classical matroid~$M\otimes_R \Frac(R)$,
or equivalently the rank of the free part of $M(\emptyset)$.

\begin{proposition}\label{prop:P ops}\noindent
\begin{enumerate}\renewcommand{\theenumi}{\alph{enumi}}
\item If $M$ is a matroid over~$R$ on ground set~$E$, then
\[P(M\setminus a) = \pi_a(P(M)\cap\{x_a = 0\}) \quad\mbox{and}\quad
P(M/a) = \pi_a(P(M)\cap\{x_a = 1\}),\]
where $\pi_a:\R^E\times\R^2\to\R^{E\setminus a}\times\R^2$
omits the $x_a$ coordinate.
\item In the same setting,
\[P(M^*) = s(P(M))\]
where $s$ is the involution on~$\R^E\times\R^2$ with
\[s(x_1,\ldots,x_n,i,y)=(1-x_1,\ldots,1-x_n,x_1+\cdots+x_n-r+i,y),\]
and $r$ is the generic rank of~$M$.
\item If $M_1$ and $M_2$ are two matroids over~$R$ on ground sets
$E_1$ and~$E_2$, then
\[P(M_1\oplus M_2) = \iota_1(P(M_2))+\iota_2(P(M_2))\]
where $\iota_i:\R^{E_i}\times\R^2\to\R^{E_1\cap E_2}\times\R^2$
are the canonical inclusions, both being the identity on the second factor.
\end{enumerate}
\end{proposition}

\begin{proof}
These are all immediate consequences of the transformations
affected by these matroid operations on the invariants $t_i(A)$.
For (a) this is no transformation at all, only
tracking the necessary map from subsets of $E\setminus\{a\}$ to those of~$E$.
For (b), the transformation is equation~\eqref{eq:dual}.
For (c) the remaining ingredient is the observation that if $N$ and $N'$ are $R$-modules, then
\[t_i(N\oplus N') = \min_{j+j'=i} t_j(N) + t_{j'}(N').\]
This holds because a generic quotient of~$N\oplus N'$ by $i$ elements
is a quotient by generators of some of its cyclic summands,
and the decomposition into cyclic summands can be chosen so that each one
is a summand of $N$ or of~$N'$.
\end{proof}

Two matroids over fields appear as base changes of a matroid $M$ over~$R$,
namely $M\otimes_R \Frac(R)$, called the \emph{generic matroid} in \cite[\S2]{FM},
and $M\otimes_R R/\m$, unnamed in that work.
These matroids appear most easily in $P(M)$ in the guise not of their
matroid basis polytopes, as above, but of their \emph{spanning set polytopes},
dual to the aforementioned independent set polytopes.
If $L$ is a usual matroid, its spanning set polytope is
\[P_{\rm span}(L) = \conv\{e_S : \mbox{$S$ a spanning set of~$L$}\}\subseteq\mathbb R^E.\]
The two are closely related:
$P(L)$ is the face of $P_{\rm span}(L)$ minimising the sum of the coordinates,
while $P_{\rm span}(L)$ is recovered by the Minkowski sum
\[P_{\rm span}(L) = (P(L) + \cone\{e_i : i\in E\})\cap [0,1]^E.\]

\begin{proposition}
Let $M$ be a matroid over~$R$ on ground set~$E$.
\begin{enumerate}\renewcommand{\theenumi}{\alph{enumi}}
\item The spanning set polytope of $M\otimes_R R/\m$
is $P(M) \cap (\R^E\times\{(0,0)\})$.
\item The spanning set polytope of $M\otimes_R \Frac(R)$
is $\pi\big(P(M) \cap (\R^E\times\{0\}\times\R)\big)$.
\end{enumerate}
\end{proposition}

Here $\pi$ is again, as in Theorem~\ref{thm:polyhedron},
the projection of $\R^E\times\R\times\R$ onto its first two factors.

\begin{proof}
If $L$ is a usual matroid viewed as a matroid over a field,
then the spanning sets of~$L$ are the sets mapped to the zero module.

The tensor product of an $R$-module $N$ with~$R/\m$ is zero
if and only if $N$ itself is zero, or equivalently $t_0(N)=0$.
So the spanning sets of $M\otimes_R R/\m$ are those contributing
vertices $(e_A,0,0)$ to~$P(M)$.

The tensor product of $N$ with~$\Frac(R)$ is zero
if and only if $N$ has no free cyclic summand, or equivalently $t_0(N)<\infty$.
Therefore, the spanning sets of $M\otimes_R \Frac(R)$ are those
contributing any vertex of form $(e_A,0,y)$ to~$P(M)$,
as only if $t_0(M(A))=\infty$ will no such vertex appear.
\end{proof}

\subsection{Faces}

We proceed to discuss the faces of $P(M)$.
If $F$ is a face of $P(M)$, let $\rec(F)$ be its recession cone.
We will classify the faces of~$P(M)$  according to the value
of $\rec(F)$.  There are four possibilities, namely the faces of
$$\rec(P(M))=\cone\{(\underline 0,1,0),(\underline 0,0,1)\}.$$

\begin{proposition}\mbox{}
\begin{itemize}
\item[(a)] The faces $F$ with $\rec(F)=\cone\{(0,1,0),(0,0,1)\}$
are exactly intersections of $P(M)$ with collections of hyperplanes of form $\{x_a=0\}$ or~$\{x_a=1\}$.
That is, the set of such faces is, up to translation,
the set of $P(N)$ for $N$ a minor of~$M$.

\item[(b)] The faces $F$ with $\rec(F)=\cone\{(0,1,0)\}$ are
the intersections of the faces in~(a) with $\{t=0\}$.

\item[(c)] The faces $F$ with $\rec(F)=\cone\{(0,0,1)\}$, when projected
to the $x$ coordinates, form the regular subdivision of the cube induced
by the corank function of the generic matroid.

\item[(d)] Among the faces $F$ with $\rec(F)=\{0\}$, i.e.\ the bounded faces,
each \emph{facet}, when projected to
the $x$ coordinates, yields a facet of the regular subdivision
of the cube induced by
\[S\mapsto \length(M(S)\otimes R/I)\]
for $I$ a nonzero proper finitely generated ideal of~$R$.
The ideal $I$ can be read off the slope of the intersection of
the defining hyperplane with the 2-space $(0,*,*)$.
The previous two cases can be seen as the limiting cases $I=R$, respectively $I=0$.
\end{itemize}
\end{proposition}

Proposition~\ref{prop:corank subdiv} below gives more detail on
the structure of the regular subdivision in~(c).  Its facets include
the independent set polytope ($S$ being the set of all loops)
and the spanning set polytope ($S$ being the set of all non-coloops).

We remark that a complete description of the faces of types (c) or~(d) seems to be far from obvious.
Indeed, a characterisation of these faces with good enough control to count them
is not to be expected, as this would provide an answer to
Speyer's $f$-vector conjecture \cite{Speyer}, which has proven to be hard.

\begin{proof}
Each face $F$ of type~(a) is determined by its image $\pi_E(F)$, where $\pi_E:\R^E\times\R^2\to\R^E$ is the projection onto the first factor.
Moreover, $\pi_E(F)$ is a face of $\pi_E(P)$.
In fact let $F'$ be the (smallest) face of $\pi_E(P)$ containing $\pi_E(F)$.
The codimension of $F$ in~$P$ equals the codimension of $\pi_E(F)$ in~$\pi_E(P)$,
so that $\pi(F)$ and~$F'$ have the same dimension and hence coincide.
The projection $\pi_E(P)$ is the unit cube $[0,1]^E$,
so $\pi_E(F)$ is the intersection of $\pi_E(P)$ with hyperplanes of equation $\{x_d=0\}, d\in D$
and $\{x_c=1\}, c\in C$ for some $D, C \subseteq E$.
Thus by the previous proposition, $F$ is (a translate of) the polytope of the minor in which the elements of $D$ have been deleted and the elements of $C$ have been contracted.

The same reasoning holds for faces $F$ of type~(b), after intersecting $P$ with the hyperplane of equation $\{t=0\}$ in which these faces are contained.

If $F$ is a face of type~(c), reasoning as before we see that $\pi(F)$ is a bounded face of $\pi(P)$. Since the corank of a set $A$ in the generic matroid is equal to the minimum of the set
$\{i \mid (e_A, i)\in\pi(P)\}$, we have that $\pi(P)$
is the regular subdivision of the cube induced by the corank function of the generic matroid.

Finally let $F$ be a face of type~(d). Let us assume that $F$ is a facet, i.e.\ the intersection of $P$ with an hyperplane $H$. The intersection of $H$ with the plane of equation $\{\underline x= \underline 0 \}$ and coordinates $(i,t)$ is a line of equation $t=-v i+c$, where $c$ is the length of the module $M(A)\otimes R / \val^{-1}([v, +\infty])$.
\end{proof}

\begin{proposition}\label{prop:corank subdiv}
Let $M$ be a (usual) matroid on ground set~$E$ with corank function $\cork$.
The facets of the regular subdivision of $[0,1]^E$ induced by $\conv\{(e_A, \cork A) : A\subseteq E\}$ are
\[\conv\{e_{S\cup A\setminus B} : A\subseteq E\setminus S, B\subseteq S,
\cork(S\cup A\setminus B) = \cork(S)-|A|\}\]
as $S\subseteq E$ ranges over the cyclic flats of~$M$.
\end{proposition}

\begin{proof}
We may suppose $M$ is connected.
The general result follows from the result applied to each connected component.

Let $C$ be the regular subdivision in question.
Then all vertices of~$C$ are lattice points by construction, and
every edge of~$C$ has direction vector of form $e-e'$ for $e,e'\in\{e_a : a\in E\}\cup\{0\}$.
With this paper's machinery this follows easiest from Theorem~\ref{thm:polyhedron}(v),
as there is a matroid $M'$ over any suitable valuation ring $R$ consisting entirely of
free $R$-modules of ranks given by~$\cork$, and the lifted polyhedron of~$C$
is the face of~$P(M')$ where the last coordinate is zero.

The constraints on edge directions imply that in the image $\widetilde C$ of~$C$ under the map
$\R^E\to\R^E\times\R$, $x\mapsto(x,-\sum_{a\in E}x_a)$,
all faces are generalised permutahedra.
Let $F$ be a facet of~$C$.
Then $F$ contains a full dimensional unimodular (lattice) simplex.
This can be constructed as follows:
the image $\widetilde F$ of $F$ in~$\widetilde C$ contains a translate of a connected matroid polytope $P$
(one can be obtained by intersecting $\widetilde F$ with enough integer translates of coordinate halfspaces;
the connectedness follows from the fact that $\dim F=|E|$).
Then any vertex $v$ of~$P$, together with
the vertices giving a spanning tree for the (bipartite) exchange graph at the basis corresponding to~$v$,
gives a unimodular simplex in $\widetilde F$,
whose projection to~$F$ is the simplex sought.

Let $f$ be the affine function on~$[0,1]^E$ whose graph contains the facet
of the lifted polyhedron of~$C$ projecting to~$F$,
so that $f(e_A)\leq\cork(A)$ for all points $e_A$, with equality on and only on the vertices of~$F$.
Pick a vertex $e_T$ of~$F$.  Then we may write
\[f(e_A) = \cork(T) + \sum_{a\in A} c_a - \sum_{b\in T} c_b\]
where the constants $c_a : a\in E$ are the derivatives of~$f$ in the coordinate directions,
and are therefore independent of~$T$.
As $F$ contains a full dimensional unimodular simplex and the corank function is integer-valued,
all of the $c_a$ are integers.
For any element $a\not\in T$, we have $f(e_{Ta})\leq\cork(Ta)\leq\cork(T)$, implying $c_a\leq0$.
Dually, if $b\in T$ we have $f(e_{T\setminus b})\leq\cork(T\setminus b)\leq\cork(T)+1$, implying $c_b\geq-1$.
Being a facet, $F$ cannot be confined to any of the subspaces of form $\{x_a=0\}$ or $\{x_a=1\}$;
therefore, for any given $a\in E$, a vertex $e_T$ can be chosen to arrange either $a\in T$ or $a\not\in T$.
This implies that $c_a\in\{0,-1\}$.

Let $S = \{a\in E : c_a = 0\}$.
Then $f(e_A)$ equals $-|A\cap (E\setminus S)|$ up to a global additive constant,
so $e_A$ is a vertex of~$F$ when and only when
\[g(e_A) = \cork(A) - |A\cap(E\setminus S)|\]
is minimised.  At $A=S$, this attains the value $g(e_S) = \cork(S)$.
This is in fact the overall minimum:
with $A=S$ as a reference point, removing elements from~$A$ will not decrease corank,
and adding an element can decrease corank by at most unity,
the effect of which is outweighed by the $-|A\cap(E\setminus S)|$ term.
The set in the statement of the proposition is the set of all sets $T$ such that
\[\cork(S) = \cork(T) - |T\cap(E\setminus S)|,\]
which are the minima of $g$ and hence the vertices of~$F$.

Finally, $S$ is a cyclic flat.  If it weren't a flat,
then for some $a\not\in S$ we would have $\cork(Sa) = \cork(S)$,
so that $g$ could not then be minimised at $e_A$ for any $A\ni a$,
implying $F\subseteq\{x_a=0\}$.  The dual argument shows $S$ is cyclic.
Conversely, if $S$ is a cyclic flat, then the minima of $g$ include $e_S$
and all vertices of $[0,1]^E$ adjacent to it,
so these minima do determine a facet of~$C$.
\end{proof}

\section{Tropical linear spaces}\label{sec:TLS}

We have already invoked valuated matroids several times in our discussion
of matroids over a valuation ring~$R$.
In this section we begin to look at these valuated matroids geometrically,
as corresponding to \emph{tropical linear spaces}.

As is natural, we will take the semifield of definition of our tropical objects
to be $\pm\val R\cup\{\infty\}$.
Tropical addition is minimum (not maximum).
Thus, a tropical polynomial is the minimum of a collection
of linear forms with integer coefficients.
Such a polynomial \emph{vanishes} if the minimum is obtained
simultaneously by two or more of the terms.
We will write simply e.g.\ that ``$\min\{\ell_1,\ldots,\ell_n\}$ vanishes'',
without specifying ``tropically''.

Let $n$ be an integer.
\begin{definition}
\begin{itemize}
\item A vector $p = (p_A)$ of elements of $\pm\val R\cup\{\infty\}$
indexed by all subsets $A\subseteq E$
is a \emph{tropical flag Pl\"ucker vector} if,
for all sets $A,B\subseteq[n]$ with $|A|\geq|B|+2$,
\[\min_{a\in A\setminus B}p_{A\setminus a}+p_{B\cup a}\mbox{ vanishes},\]
and moreover, for each $0\leq r\leq n$, there is at least one
set $A$ of cardinality $r$ such that $p_A\neq\infty$.

\item If $p$ is indexed only by the subsets of~$E$ of cardinality~$r$,
satisfies these relations when $|A|-1=|B|+1=r$,
and not all entries of~$p$ are $\infty$,
we call it a \emph{tropical Pl\"ucker vector (of rank $r$)}.
\end{itemize}
Every tropical Pl\"ucker vector is a projection $p_r$
of some tropical flag Pl\"ucker vector $p$ to the coordinates indexed by size-$r$ sets.
\end{definition}
To each tropical Pl\"ucker vector $p$ on~$E$ of rank~$r$
there is associated a tropical linear space $L(p)$.
To each tropical flag Pl\"ucker vector
corresponds a \emph{full flag} of linear spaces, one of each rank,
namely the collection of $L(p_r)$ as $r$ ranges from 0 to~$n$;
these linear spaces are said to be \emph{incident}.
The tropical linear spaces $L(p)$ and $L(q)$ are equal if and only if
$p$ and $q$ differ by a global additive (i.e.\ tropically multiplicative) constant.
For our purposes this may be taken as the definition of ``tropical linear space''.

The set (which is a projective tropical prevariety)
of all tropical linear spaces on~$E$ of rank~$r$
is called the \emph{Dressian} $Dr(r,n)$, where as above $n=|E|$.
The set of complete flags of tropical linear spaces
is the flag Dressian $FD(n) = FD(1,\ldots,n-1;n)$ defined by Haque \cite{Haque}.

In the literature,
tropical linear spaces are sometimes defined with the additional condition
that $p_A\neq\infty$ for all~$A$.  Under this assumption one can get away
with fewer equations:

\begin{lemma}\label{lem:three term TLS}
If a vector $(p_A)_{|A|=r}$ with no infinite components satisfies
the \emph{three-term Pl\"ucker relations},
namely that for all sets $|A|=r-2$ and all $b,c,d,e\not\in A$,
\[\min\{p_{Abc}+p_{Ade},p_{Abd}+p_{Ace},p_{Abe}+p_{Acd}\}\mbox{ vanishes},\]
then $(p_A)$ is a tropical Pl\"ucker vector.
\end{lemma}

\begin{proof}
The definition of tropical linear space given by Speyer in~\cite{Speyer}
uses only the three-term Pl\"ucker relations.
The equivalence of this definition and our own
is proven by Speyer's Proposition~2.2 and our Theorem~\ref{thm:matroid polytope}.
\end{proof}

\begin{lemma}\label{lem:three term incidence}
Let $p$ and $q$ be tropical Pl\"ucker vectors of respective ranks
$r$ and~$r+1$ on ground set~$E$.
If $p$ and $q$ have no infinite entries and satisfy the three-term incidence relations,
namely that for all sets $|A|=r-1$ and all $b,c,d\not\in A$
\[\min\{p_{Abc}+p_{Ad},p_{Abd}+p_{Ac},p_{Acd}+p_{Ab}\}\mbox{ vanishes},\]
then $p$ and $q$ are incident.
\end{lemma}

\begin{proof}
Define a new vector $f$ on an augmented ground set $E^\ast = E\cup\{\ast\}$,
indexed by size $r+1$ subsets of~$E^\ast$, so that
\[f_A = \begin{cases}
p_{A\setminus\ast} & \ast\in A \\
q_A & \ast\not\in A.
\end{cases}\]
The three-term Pl\"ucker relations for $f$ are either three-term Pl\"ucker relations
for one of $p$ or~$q$, or are three-term incidence relations between the two.
Since these relations are all satisfied, Speyer's characterisation
implies that $f$ is a tropical linear space.
The general incidence relations between $p$ and~$q$ are general Pl\"ucker relations on~$f$,
completing the proof.
\end{proof}

For $M$ a matroid over~$R$ on ground set~$E$, $|E|=n$,
and $r,i$ a pairs of nonnegative integers such that $r+i$ exceeds the generic rank of~$M$,
let $t_{r,i} = t_{r,i}(M)$ be the vector indexed by size $r$ subsets of~$E$
given by $(t_{r,i})_A = t_i(M(A))$.
The restriction on~$r+i$ ensures that not all the $t_i(M(A))$ are infinite.

\begin{lemma}\label{lem:closure of finite}
With setup as above,
each vector $t_{r,i}$ is a tropical Pl\"ucker vector of rank~$r$.
In fact,
\[L(t_{0,i_0}) \subseteq L(t_{1,i_1})\subseteq\cdots\subseteq L(t_{n,i_n})\]
is a full flag of tropical linear spaces whenever
$i_0, \ldots, i_n$ are indices of with $i_r-i_{r+1}\in\{0,1\}$ for each~$r$.
\end{lemma}

\begin{proof}
Property~(D4) of~$M$ is equivalent to the
three-term Pl\"ucker relations for each $t_{r,i}$, while
property~(D3) is equivalent to the assertion that
$L(t_{r,i})$ and $L(t_{r+1,i})$ satisfy the three-term incidence relations
for all valid $r$ and $i$,
and similarly (D3$'$) is equivalent to the assertion that
$L(t_{r,i})$ and $L(t_{r+1,i-1})$ satisfy the three-term incidence relations.

For each ideal $I$ of~$R$, the base change $M\otimes R/I$
can be viewed as a matroid over~$R$, via the $R$-module structure of~$R/I$.
Given a descending sequence $(I_k)$ of ideals whose intersection is zero,
$t_i(M\otimes R/I_k(A))$ is a sequence of finite numbers
which stabilises at $t_i(M(A))$ if the latter is finite,
and increases without bound otherwise.
That is, if we let $t_{r,i}^k$ be the vector with entries
$t_i(M\otimes R/I_k(A))$ for $|A|=r$,
then for any valid indices $(i_r)$ the sequence of tuples
$(t_{i,r_i}^k)$ converges to $(t_{i,r_i})$
in the product of tropical projective spaces containing them.

By Lemma~\ref{lem:three term TLS},
every $(t_{i,r_i}^k)$ is a tropical linear space,
and by Lemma~\ref{lem:three term incidence},
$(t_{i,r_i}^k)$ is always incident to $(t_{i+1,r_{i+1}}^k)$.

In fact, by inductive use of Lemma~\ref{lem:three term incidence} we get
all the incidences in our flags, as follows.
Let $(i,r)$ and $(i',r')$ be pairs of indices that appear together
in the chain of indices for some full flag,
where without loss of generality $r'>r$ (the case of inequality being vacuous),
so $0\leq i-i'\leq r'-r$ and at least one of these inequalities is strict.
We suppose the former is, the proof for the latter being symmetric under matroid duality.

Consider the enlarged ground set $E' = E\amalg G\amalg Z$ where
$|G|=(r'-r)-(i-i')$ and $|Z|=(i-i')-1$.
We construct two valuated matroids $M_1$ and $M_2$ on~$E'$,
of respective ranks $r'-1$ and $r'$.
Set $M_1(A) = t_{i+|A\cap G|}(M(A\cap E))$ when $|A|=r'-1$,
and $M_2(A)$ to be given by the same formula when $|A|=r$.
The tropical Pl\"ucker relations for $M_1$ and $M_2$ hold by induction on $r'-r$.
Each one of the three-term incidence relations between $M_1$ and $M_2$
follows from one of the nine properties, from (D4) to (D0$''$),
discussed in Section~\ref{sec:D4}.
Lemma~\ref{lem:three term incidence} then implies that $M_1$ and~$M_2$ are incident.
Among their incidence relations are the incidence relations between
$(t_{i,r}^k)$, which appears as a minor of~$M_1$, and
$(t_{i',r'}^k)$, which appears as a minor of~$M_2$.
And the next case of the inductive hypothesis holds:
the incidence relations between $M_1$ and~$M_2$
together with the Pl\"ucker relations on $M_1$ and~$M_2$
make up the Pl\"ucker relations for the new
valuated matroid to be constructed next, with $M_1$ as deletion
and~$M_2$ as contraction by the new element added to~$Z$.

Lastly, since $FD(n)$ is a tropical prevariety, it is closed,
so we can conclude that because the flags formed by the $(t_{i,r_i}^k)$
lie in $FD(n)$ for all~$k$, the flags formed by their limits $(t_{i,r_i})$ do as well.
\end{proof}

\begin{lemma}\label{lem:incidence implies pluecker}
Let $p$ be a tropical Pl\"ucker vector on $[n]$ of rank~$r-1$,
and $q$ a vector indexed by size $r$ subsets of $[n]$
such that $p$ and $q$ satisfy the tropical incidence relations.
Then $q$ is also a tropical Pl\"ucker vector.
\end{lemma}

Since the tropical incidence relations are self-dual,
Lemma~\ref{lem:incidence implies pluecker} also holds if $p$ is indexed by size $r+1$ subsets.

\begin{proof}
Towards the contrapositive, suppose that $q$ was not a tropical Pl\"ucker vector.
By Theorem~\ref{thm:matroid polytope}(b),
this implies that the corresponding lifted hypersimplex
\[P(q) = \conv\{(e_A,q_A) : A\in\binom{[n]}r\} + \cone\{(\underline 0,1)\}\]
has an edge $e = \conv\{(e_A,q_A),(e_B,q_B)\}$ for $A,B\in\binom{[n]}r$
for which $e_A-e_B$ is not of the form $e_i-e_j$.
Now form the lifted polyhedron corresponding to $p$ and $q$ together,
\[P(p,q) = \conv(\{(e_A,p_A) : A\in\binom{[n]}{r-1}\}\cup\{(e_A,q_A) : A\in\binom{[n]}r\})
+ \cone\{(\underline 0,1)\}.\]
Since $P(q)$ is the face of $P(p,q)$ maximising the sum of the non-last coordinates,
$e$ remains an edge of $P(p,q)$.  Some 2-dimensional face $f$ of~$P(q)$ includes $e$
and a vertex not on~$P(q)$.  The vertices of~$f$ which are not vertices of~$e$
do not lie on~$P(q)$ and are therefore of the form $(e_A,p_A)$,
making them vertices of the valuated matroid polyhedron $P(M_p)$,
where $M_p$ is the valuated matroid with valuation data $p$.
The convex hull of these vertices is $P(M_p)\cap f=:g$, whose dimension is at most~1.

If $g$ were an edge it would be
parallel to~$e$, contradicting Theorem~\ref{thm:matroid polytope}(b)
for the valuated matroid polyhedron $P(M_p)$.
So $g$ has dimension 0, i.e.\ is a single vertex $\{(e_C,p_C)\}$
for some $C\in\binom{[n]}{r-1}$.
At least one of $A$ or~$B$ must not be a superset of~$C$,
or else $e_A-e_B$ is in a direction $e_i-e_j$.
Suppose without loss of generality that $A\not\supseteq C$.
But then $p_C+q_A$ is a term of some tropical incidence relation $\mathcal R$,
and the existence of the edge
$\conv\{(e_C,p_C),(e_A,q_A)\}$ of~$P(p,q)$ contradicts $\mathcal R$.
To wit, consider the subpolyhedron $R$ of~$P(p,q)$
whose vertices are the ones corresponding to entries appearing in~$\mathcal R$.
The tropical equation $\mathcal R$ says exactly that
the midpoint $\frac12(e_C,p_C)+\frac12(e_A,q_A)$,
provided that it is on the boundary of~$R$ and in particular minimises its last coordinate,
should not lie on an edge of~$R$ but on a higher-dimensional face,
because some other pair of vertices $(e_{C'},p_{C'})$
and $(e_{A'},q_{A'})$ have the same midpoint.
\end{proof}

\begin{corollary}
The set of $r$-dimensional tropical linear spaces containing a given $(r-1)$-dimensional
space, as a subset of the Dressian, is tropically convex.
\end{corollary}

Again, dually, the same is true for $r$-dimensional spaces contained in a given
$(r+1)$-dimensional space.

\begin{proof}
Since the incidence relations are linear in the $q$ variables
when the $p$ variables are fixed, each one cuts out a tropical linear space
in the space of vectors $q$, which is tropically convex.
By Lemma~\ref{lem:incidence implies pluecker}, this intersection is
contained in the Dressian; the equations of the Dressian need not be imposed separately.
The corollary follows because intersections of tropically convex sets are tropically convex.
\end{proof}

\section{Parameter spaces}\label{sec:parameter}
In this section we discuss the parameter space of the set of all matroids over
a fixed valuation ring~$R$.
We will take our ground set in this section to be
$E=[n]=\{1,\ldots,n\}$ for some natural~$n$, to line up with usual practice
when speaking of flag varieties.  This entails no loss of generality.

\subsection{Dressians}
The Dressian itself is a complicated space: for instance, we
only know its dimension asymptotically \cite[Theorem 31]{JoswigSchroeter:2017}, and to find an exact value
seems a difficult problem.
Since matroids over~$R$ can be seen as enrichments of tropical linear spaces,
it will be at least as difficult to gain
careful control of the space of matroids over~$R$.
But if the Dressian is taken to be a known object,
then this section provides an explicit description of
the parameter space of matroids over~$R$ in Proposition~\ref{prop:slice}.

We first define some restricted subsets of matroids over~$R$.
We say that a matroid $M$ over $R$ is \emph{essential} if no nontrivial projective module is a direct summand of $M(E)$.
For integers $k$ and $\ell$, let $\Mat{R}(r,n;k,\ell)$ be the set of
essential matroids $M$ over~$R$
on the ground set $[n]$ and of generic rank~$r$
such that $M(\emptyset)$ can be presented with at most $r+k$ generators
(of which $r$ will be dedicated to the free summand, leaving $k$ others)
and $M(E)$ needs at most $\ell$ generators.
For example, when $\ell=0$, the matroids over~$R$ in question
are the spannable ones, discussed in Section~\ref{sec:spannable}.
Clearly $\bigcup_{k,\ell}\Mat{R}(r,n;k,\ell)$ is the set of all
essential matroids $M$ over~$R$ on~$E$ of generic rank~$r$.
The essentiality hypothesis is not a significant hindrance, because  (as proved in \cite[Lemma 2.5]{FM} ) any matroid
with these parameters is the direct sum of an essential one
and a zero-element matroid for a free module.

\begin{proposition}\label{prop:slice}
The set $\Mat{R}(r,n; k,\ell)$ is the intersection of the Dressian
$Dr(n+k,2n+k+\ell)$ with a linear space.
\end{proposition}

The embedding $\xi:\Mat{R}(r,n; k,\ell)\to Dr(n+k,2n+k+\ell)$
is the same one exploited in the proof of Lemma~\ref{lem:closure of finite}, and is as follows.
Partition $[2n+k+\ell]$ into $E=[n]$ and two further sets
$G$ and~$Z$ of respective cardinalities $r+k$ and $n-r+\ell$.
Then the matroid $M$ over~$R$ is sent to
the tropical Pl\"ucker vector given by
\[p_A = t_{|A\cap G|}(M(A\cap [n]))\]
for any $A\subseteq[2n+k+\ell]$ of size~$n+k$.
In the language of Section~\ref{sec:D4}, $G$ and $Z$ are sets
of generic and zero elements.

The linear space $W$ in question is described by the equations
$p_A = p_B$ when $|G\cap A| = |G\cap B|$ and $|Z\cap A| = |Z\cap B|$,
and all the $p_A$ are equal when $|Z\cap A|=0$ or when $|G\cap A| = r = |G|$.
This is a linear space both classically and tropically.

\begin{proof}
By the arguments of Section~\ref{sec:D4}, the Pl\"ucker relations
for the $p_A$, for the various compositions that $A$ might have
in terms of elements of $[n]$ and $G$ and~$Z$,
are equivalent to statement (D4) and its various genericisations and zeroisations.
Together with (TS) these are an axiom system for matroids over~$R$.
So $\Mat{R}(r,n;k,\ell)$ equals the set of Pl\"ucker coordinates of points in $Dr(n+k,2n+k+\ell)$
that come from a collection of data $t_{i}(A)$ and satisfy~(TS);
these two conditions impose the two kinds of equalities defining~$W$.
Since all $p_A$ for $Z$ disjoint from~$A$ are zero,
no two distinct matroids $M$ are identified by the tropical projectivisation.
\end{proof}

\subsection{Bott-Samelson varieties}
Greater insight into the structure of the space $\Mat{R}(r,n;k,\ell)$
comes by relating it to tropical versions of
Bott-Samelson varieties.
We first recall what these are in algebraic geometry.
Let $\mathbb K$ be the fraction field of~$R$.
Let $\Fl(\mathbb K^n)$ be the full flag variety,
parametrising complete flags of linear subspaces of~$\mathbb K^n$.
Fix a complete flag $\mathcal F$ in $\mathbb K^n$,
and let $w = s_{i_1}\ldots s_{i_\ell}$ be a word in the symbols $s_1,\ldots,s_{n-1}$,
i.e.\ $i_k\in\{1,\ldots,n-1\}$ for each $k$.
The \emph{Bott-Samelson variety} of $w$ is
\begin{multline*}
Z_w = \{(\mathcal F_0,\ldots,\mathcal F_\ell) \in \Fl(\mathbb K^n)^{\ell+1} :
\mathcal F_\ell = \mathcal F, \\
\mbox{$\mathcal F_k$ and $\mathcal F_{k+1}$ agree except in the $i_k$-dimensional space}\}.
\end{multline*}

The variety $Z_w$ projects to the complete flag variety by retaining only
the flag $\mathcal F_0$.
The image of this projection is a Schubert variety~$X_v$,
where $v$ is the element of the Weyl group~$A_{n-1}$
such that $\pm T_v = T_{i_1}\cdots T_{i_\ell}$ in the
0-Hecke algebra, in the notation of Norton~\cite{Norton}.
If $w$ is a reduced word for $v$ in~$A_{n-1}$,
then $Z_w$ is a resolution of singularities of~$X_v$.
In this case the map $Z_w\to X_v$ is one-to-one on the locus
where $\mathcal F_k$ and $\mathcal F_{k+1}$ do not agree in the $i_k$-dimensional space
for any~$k$.

For a suitable choice of~$w$,
the tapestries of flags that give elements of the Bott-Samelson variety
also describe the incidences between all the various spaces $L(t_{i,r})$
contained in a matroid over~$R$.
This section is dedicated to translating this observation into
a description of the parameter space of matroids over~$R$,
which is summarised at a high level in Theorem~\ref{thm:BS}.
The theorem is (somewhat awkwardly) inexplicit in two ways,
reflecting familiar tropical phenomena:
\begin{enumerate}
\item The separation between the tropical Grassmannian and the Dressian.
Matroids over rings contain the data of many tropical linear spaces;
in Theorem~\ref{thm:BS}
we only bound their parameter space between the ``Grassmannian and Dressian
versions'' of a construction with linear spaces.
\item Positivity requirements.
These are of the same general flavour as the conditions which
prevent $Dr(2,n)$ from literally being the space of metric trees on~$n$ leaves:
the leaf edges of a tree should not have negative length.
This is the role of the cone in Theorem~\ref{thm:BS}.
\end{enumerate}

To wit, our ``Dressian version'' of $Z_w$ is the \emph{Bott-Samelson Dressian}
\begin{multline*}
ZD_w = \{(\mathcal F_0,\ldots,\mathcal F_s) \in FD(n)^{s+1} :
\mathcal F_s = \mathcal F, \\
\mbox{$\mathcal F_k$ and $\mathcal F_{k+1}$ agree except in the $i_k$-dimensional space}\},
\end{multline*}
where now the $\mathcal F_i$ are flags of tropical linear spaces.
Like $FD(n)$, the set $ZD_w$ is not usually a tropical variety.

Say that an \emph{exceptional pair} for a matroid $M$ over~$R$ is
a pair $(s,i)$ with $0<s<n$ and $i\geq0$
such that there exists $\mu\in\val R$
with $t_i(M(A)) = t_{i+1}(M(A))+\mu$ for all sets $A\subseteq[n]$ of size~$s$.
That is, exceptional pairs index the instances in which
$t_{s,i}$ and $t_{s,i+1}$ are equal up to a global additive scalar $\mu$.
Matroids with an exceptional pair are those that
belong to the exceptional locus of the map $Z_w\to X_v$.
Let the \emph{exceptionality} of~$M$ be the number of its exceptional pairs.

\begin{example}
If $(R,\mathfrak m)$ is a discrete valuation ring,
then a matroid $M$ over~$R$ of exceptionality 1 on ground set $\{1,2\}$ is given by
$M(\emptyset)=R\oplus R/\mathfrak m^2$,
$M(1)=M(2)=R/\mathfrak m^3$, and
$M(12)=R/\mathfrak m$.
The exceptional pair is $(s,i)=(1,0)$,
for which $t_{1,0}=(3,3)$ and $t_{1,1}=(0,0)$ represent the same point in $\mathbb T\mathbb P^1$.
Foreshadowing Theorem~\ref{thm:BS},
this is related to the existence of another matroid $M'$ over~$R$
identical to $M$ except that $M(1)=M(2)=R/\mathfrak m^2$,
i.e.\ differing only in adding a global scalar to $t_{1,0}$.
\end{example}

For fixed parameters $r,n,k,\ell$, abbreviate the word
$s_is_{i+1}\cdots s_j$ by $x_{ij}$, and let $w$ be the word
\begin{equation}\label{eq:w}
w = \underbrace{x_{r,n-1}\,x_{r-1,n-1}\cdots x_{*,n-1}}_{\ell+1}
\,x_{*,n-2}\cdots x_{2,*}\,
\overbrace{x_{1,*}\cdots x_{1,n-r-k+\ell+1}\,x_{1,n-r-k+\ell}}^{k+1}.
\end{equation}
The way this is written calls for clarification.  The first and second indices
of the sequence of factors $x$ should be conceived of as
forming noninteracting sequences, the first decreasing from $r$ to~1
and then taking the value 1 another $k$ times, and the second
first repeating $\ell$ times the value $n-1$ and then decreasing
from $n-1$ (once more) to~$n-r-k+\ell$.
The subwords consisting of the first $\ell+1$ and the last $k+1$ factors
may or may not overlap (the display is written as if they do not;
in this way its look may be deceiving).
The asterisked indices are therefore simply piecewise-linear functions
with concrete values.  We have only written them as~$*$
to prevent the displayed line from being too ungainly.

If $k=\ell=0$, then this word $w$ is the longest among
the shortest representatives of cosets of the parabolic subgroup
$S_{r}\times S_{n-r}\subseteq S_{n}$,
and the corresponding Schubert variety $X_w$ is the smallest one to which
the restriction of the projection $\Fl(\mathbb K^n)\to\Gr(r,\mathbb K^n)$ is surjective.
Note also that incrementing both $k$ and $\ell$ at the same time
multiplies $w$ by a Coxeter element.

Finally, define
\[D(k,\ell) = k \ell - \binom{\min(0,\ell-r)+1}2 - \binom{\min(0,k-n+r)+1}2.\]
We may as well assume that $|k-\ell|\leq n$ since the quotient by a cyclic module can add or remove at most one torsion summand.
When either $k$ or~$\ell$ is large with respect to~$n$, implying by assumption that both are,
then $D(k,\ell)$ is of the order of~$kn$.
Of course, when $k$ and~$\ell$ are at most $n-r$ and~$r$ respectively,
we get $D(k,\ell) = k \ell$.

\begin{theorem}\label{thm:BS}
Assume that the residue field $R/\val^{-1}((0,\infty])$ is infinite.

Let $w$ be the word in~\eqref{eq:w}.
There is a space $Z'_w$ with $\Trop Z_w \subseteq Z'_w \subseteq ZD_w$
such that $\Mat{R}(r,n; k,\ell)$ projects piecewise linearly
onto the intersection of $Z'_w$ with a full-dimensional cone.
The fibres are polyhedral complexes of dimension at most $D(k,\ell)+k+\ell$.
\end{theorem}
Thus in the case $k=\ell=0$ of bispannable matroids, where $M([n])=0$ and $M(\emptyset)$ is free,
$\Mat{R}(r,n; 0,0)$ is piecewise linearly isomorphic to
the intersection of $Z'_w$ with a cone.

The Bott-Samelson variety $Z_w$ is embedded by a product of Pl\"ucker embeddings of Grassmannians
in an ambient product $\prod\P^{\binom nr_i}$ of projective spaces.
We can take one projective space for each dimension $0<d<n$
and each maximal run $j=k,k+1,\ldots,\ell$ of indices such that
no $i_j = d$.
We can in theory omit those runs that include $j=0$, though it will
make our statements easier to include them.
For our word~$w$, this amounts to one projective space $\P^{\binom ns-1}$
for each valid pair of indices $(s,i)$ such that $t_{s,i}$ is defined.
This is the embedding with respect to which $ZD_w$ is notionally tropicalised.
That is, $ZD_w$ is contained in the product
$\prod_{s,i}\Trop\P^{\binom ns-1}$ of tropical projective spaces
over the above set of~$(s,i)$.
Write $p_{A,i}$ for the coordinate indexed by the set $A\subseteq[n]$
of the $(|A|,i)$-th of these tropical projective spaces,
or its affine counterpart.

Let $\mathcal C$ be the cone in the tropical affine space $\prod_{s,i}\Trop\mathbb A^{\binom ns}$
of points $(p_{A,i})$ satisfying all inequalities of the following forms:
\begin{align*}
p_{Ab,i+1} + p_{A,i} &\geq p_{Ab,i} + p_{A,i+1}, \\
p_{Ab,i} + p_{A,i+2} &\geq p_{Ab,i+1} + p_{A,i+1}, \\
p_{A,i} + p_{A,i+2} &\geq 2p_{A,i+1}.
\end{align*}
Observe that $\mathcal C$ is full-dimensional.
Let $\mathcal L$ be the classical linear space in $\prod_{s,i}\Trop\mathbb A^{\binom ns}$
defined by
\begin{align*}
p_{A,i} &= p_{B,i} \quad\mbox{if $|A|=|B|$}, \\
p_{Abc,i} + p_{A,i+1} &= p_{Ab,i+1} + p_{Ac,i}.
\end{align*}
Observe that $\mathcal L$ is parametrised by
$p_{A,i} = \lambda_i + \mu_{|A|+i}$.

\begin{lemma}\label{lem:cone}\indent
The space $\Mat{R}(r,n; k,\ell)$ is contained in $\mathcal C$.
\end{lemma}

\begin{proof}
The defining inequalities of $\mathcal C$ are the properties
(D1$'$), (D1$''$), and~(D0$''$) of Section~\ref{sec:D4}.
\end{proof}

We observe that points of~$ZD_w\cap\mathcal C$ will satisfy the further inequalities
\begin{align*}
p_{Abc,i} + p_{A,i} &\geq p_{Ab,i} + p_{Ac,i}, \\
p_{Abc,i} + p_{A,i+2} &\geq p_{Ab,i+1} + p_{Ac,i+1},
\end{align*}
corresponding to (D2) and~(D2$''$).
We could just as well have included these in the definition of~$\mathcal C$.

For our word $w$ from~\eqref{eq:w}, there exists a map
\[\pi:\Mat{R}(r,n;k,\ell)\to ZD_w\]
taking the distinguished flag $\mathcal F$ to be the flag all of whose
Pl\"ucker coordinates are zero.
For $M\in\Mat{R}(r,n;k,\ell)$, the $j$-th flag $F_i$ appearing in~$\pi(M)$ is
\begin{equation}\label{eq:pi i}
\cdots\subseteq L(t_{r-1,i(j,r-1)}) \subseteq L(t_{r,i(j,r)})\subseteq L(t_{r+1,i(j,r+1)})\subseteq\cdots
\end{equation}
where the $j=0$ sequence of indices is given by
\[i(0,m) = \min(r,m) - \max(-k,m-n+r-\ell)\]
and the remaining sequences so that $i(0,m) - i_(j,m)$ is the number of
letters $s_m$ to be found in the first $j$ letters of~$w$.  In particular
the sequence indexing the last flag is
\[(i(|w|,m))_{m=1}^{n-1} = (r-1,r-2,\ldots,1,0,0,\ldots,0).\]
Recall that, by the discussion before Lemma~{lem:closure of finite},
since any $M\in\Mat{R}(r,n;k,\ell)$ is of generic rank~$r$,
none of the Pl\"ucker vectors appearing in $\pi(M)$ consist entirely of infinities,
so $\pi(M)$ is well-defined.

\begin{proposition}\label{prop:pi}
Fix parameters $r,n,k,\ell$.
The nonempty fibres of the map
\[\widehat\pi : M\mapsto(\pi(M), M(\emptyset), M([n]))\]
are bounded polytopes, of dimension bounded above
by the exceptionality of the matroids therein and by $D(k,\ell)$.
\end{proposition}

\begin{proof}
If two matroids $M_1$ and $M_2$ over~$R$ fall together under~$\widehat\pi$,
then there exists a family of constants $\lambda_{s,i}$
so that $t_{s,i}(M_2)$ is obtained from $t_{s,i}(M_1)$ by a global
scalar addition of $\lambda_{s,i}$ for all $s$ and~$i$.
Because $\widehat\pi$ remembers $M(\emptyset)$ and~$M([n])$,
we have $\lambda_{0,i}=\lambda_{n,i}=0$.
Aside from the claim that the fibres are polyhedra,
we will prove the proposition by showing that
$\lambda_{s,i}$ takes a single fixed value,
which is a function of $\pi(M_1)$ and those ``later'' $\lambda_{s',i'}$
with $(s',i')>(s,i)$ in right-to-left lexicographic order,
unless $(s,i)$ is an exceptional pair for~$M_1$
and $i<\ell$ and $s-r+i<k$.
We will further prove that even if these conditions attain, $\lambda_{s,i}$ is bounded.
The reader may check that
$D(k,\ell)$ is the number of pairs of integers $(s,i)$ satisfying $i<\ell$ and $s-r+i<k$
as well as the obvious conditions $i\geq0$ and $0<s<n$.

We first prove polyhedrality.
Given a matroid $M_1$ over~$R$, consider the fibre consisting of all $M_2$ with
$\pi(M_2) = \pi(M_1)$.
This fibre is defined by a collection of linear inequalities.
A priori it can be defined by the linear equations $\pi(M_2) = \pi(M_1)$
and the axiom system of Proposition~\ref{prop:t}.
Of the axioms (TS) is guaranteed by the parametrisation, (T0) and~(T1) are linear inequalities,
and (T2), which we have seen is equivalent to~(D2$'$),
becomes a linear inequality after imposing $\pi(M_2) = \pi(M_1)$,
since this implies $t_i(Abc)+t_{i+1}(A) = t_{i+1}(Ab)+t_i(Ac)$.

We turn to analysing the $\lambda_{s,i}$.
Let $(s,i)$ be an unexceptional pair for~$M_1$.
The sets $B$ of size $s$ where $t_i(M_1(B))$ is finite are the bases of a matroid.
Because matroids have connected exchange graphs, unexceptionality implies that
there exist two sets $Ab$ and~$Ac$ of size~$s$ differing by a single exchange such that
$t_i(M_1(Ab))-t_{i+1}(M_1(Ab))$ is unequal to $t_i(M_1(Ac))-t_{i+1}(M_1(Ac))$,
and at least one of these differences is finite.
(Note that, by axiom~(T0), $t_i(M_1(Ab))$ being finite implies $t_{i+1}(M_1(Ab))$ is finite.)
Therefore the first two terms of
\begin{multline*}
\min\big\{t_i(M_1(Ab))+t_{i+1}(M_1(Ac)), t_i(M_1(Ac))+t_{i+1}(M_1(Ab)), \\
t_i(M_1(Abc))+t_{i+1}(M_1(A))\big\}
\end{multline*}
in the instantiation of relation~(D2$'$) are unequal,
so they do not jointly achieve the minimum.
In the corresponding instantiation of relation (D2$'$) for~$M_2$,
the first two terms have each been increased by $\lambda_{s,i}$,
so they still do not jointly attain the minimum.
The third term is determined by $\pi(M_1)$ and the later values
$\lambda_{s+1,i}$ and~$\lambda_{s-1,i+1}$, both of which have subscripts greater than $(s,i)$
in right-to-left lexicographic order.
This implies that $\lambda_{s,i}$ must equal some fixed function
of $\pi(M_1)$ and the later $\lambda$ values
for the minimum to be attained twice.

Next suppose $i\geq\ell$.
Corollary~\ref{cor:i>=l} implies that
\[
t_{i+1}(M_m;A) = \min_{b\not\in A}\{t_i(M_m;Ab)\}
\]
for $m=1,2$ and any $A\subseteq[n]$ and~$i\geq\ell$.
By taking $|A|=s-1$ (which we can do outside the trivial case $s=0$),
the left sides of~\eqref{eq:sg1}
for $M=M_1$ and $M=M_2$ are seen to differ by~$\lambda_{i+1,s-1}$,
and hence the right sides must have the same difference,
proving $\lambda_{s,i}=\lambda_{s-1,i+1}$.
The dual of this argument proves that
$\lambda_{s,i}=\lambda_{s,i-1}$ if $s-r+i\geq k$.

Even in the absence of Corollary~\ref{cor:i>=l},
$\lambda_{s,i}$ can take a bounded range of values.
Let $|A|=s$ be such that $t_i(M_1(A))$ is finite.
Then $t_i(M_1(Ac))$ is finite for any $c\not\in A$,
as is $t_{i+1}(M_1(A\setminus b))$ for any $b\in A$.
Therefore in relation (D2$'$), written as
\begin{multline*}
\min\big\{t_i(M_1(A))+t_{i+1}(M_1(Ac\setminus b)), t_i(M_1(A))+t_{i+1}(M_1(Ac\setminus b)), \\
t_i(M_1(Ac))+t_{i+1}(M_1(A\setminus b))\big\},
\end{multline*}
the last term is finite.
Hence even if the first two terms are joint minima,
their counterparts in $M_2$ cannot be larger than the last term,
which bounds $\lambda_{s,i}$ above in terms of $\pi(M_1)$ and later $\lambda$ values.
Axiom~(T0) bounds $\lambda_{s,i}$ below in terms of the same.
\end{proof}

Since the residue field of~$R$ is infinite, an $n\times n$ matrix can be chosen over
the residue field with no vanishing top-aligned minors.
Lifting its entries to~$R$ arbitrarily then produced a matrix
all of whose top-aligned minors have zero valuation.
Fix such a matrix and let $\mathcal F$ be the flag whose $i$-dimensional
space is the span of its top $i$ rows.
Continue to let $w$ be the word in~\eqref{eq:w}.
The classical Bott-Samelson variety $Z_w$ for the flag $\mathcal F$
has a projective embedding
obtained by giving each of its coordinate spaces $F\ell_n$ the Pl\"ucker embedding,
under which its tropicalisation $\Trop Z_w$ lies in the ambient space of~$ZD_w$.
This is the tropicalisation meant in Theorem~\ref{thm:BS}.

\begin{proposition}\label{prop:moduli lower bound}
Fix values $(n,r,k,\ell)$ of the parameters.
That part of $\Trop Z_w$ lying in~$\mathcal C$ is a subset of~$\im\pi$.
\end{proposition}

\begin{proof}
%
Let $E$ be the Zariski closed subset of~$Z_w$ of points
$\mathcal F_\ast=(\mathcal F_0,\cdots,\mathcal F_s)$
for which $\mathcal F_j=\mathcal F_{j+1}$ for some~$j$.
Because fibres of the tropicalisation map
(i.e.\ coordinatewise valuation) are Zariski dense
\cite[Proposition~4.14]{Gub12}, \cite[Theorem~4.2.5]{OP10},
every point of $\Trop Z_w$ has a preimage under tropicalisation in~$Z_w\setminus E$.
So it is enough to show that, given $\mathcal F_\ast\in Z_w\setminus E$,
if the valuations $p_{r,i}$ of the Pl\"ucker coordinates of the $F_j$
lie in~$\mathcal C$,
then there is a matroid $M$ over~$R$ so that the $t_{r,i}(M)$ equal the $p_{r,i}$.

Constructing $M$ is a matter of choosing the global (additive) scalars
in the valued Pl\"ucker coordinates of the~$\mathcal F_j$
so that the results form a matroid over~$R$.
We do this iteratively, working along the list $\mathcal F_\ast$
from right to left, starting with the known final flag $\mathcal F_s = \mathcal F$.
As such we initiate the iterative construction by setting $t_i(M(A))$ to zero
when $i \geq \min(r,|A|) - \max(-k,|A|-n+r-\ell)$,
these being the values that are to agree with the
valuations of Pl\"ucker coordinates of~$\mathcal F$.

At a later stage of the construction, corresponding to the index $j$,
let us suppose that the values $t_i(M(A))$
that correspond to flags $\mathcal F_{j+1}$ through $\mathcal F_s$,
that is for $i\geq i(j+1,|A|)$ in the notation of~\eqref{eq:pi i},
have already been assigned.
To progress to $\mathcal F_j$
we only need to assign these values when $|A|=m$ is such that
the $(j+1)$th letter of~$w$ is $s_m$, and $i=i(j+1,m)=i(j,m)-1$.

Let $F^{m-1}\subseteq F^m\subseteq F^{m+1}$ be the $m-1$, $m$, and $(m+1)$-dimensional
spaces in~$\mathcal F_{j+1}$, respectively,
and $G$ the $m$-dimensional space in~$\mathcal F_j$, also containing $F^{m-1}$ and contained in $F^{m+1}$.
Let $v$ and~$v'$ be vectors in~$K^n$ so that $F^m = F^{m-1}+Kv$ and $G=F^{m-1}+Kv'$.
Since $G\neq F^m$, we have $F^{m+1} = F^{m-1}+Kv+Kv'$.

Now let $L$ be the linear subspace
\[F^{m-1}\!\times\!\{(0,0)\}+K(v,c,0)+K(v',0,c')\subseteq K^n\times K^2\]
where $c,c'\in K$ are constants to be chosen shortly.
The Pl\"ucker coordinates of $L$ are indexed by subsets of $[n]\cap\{\ast,\overline\ast\}$
of size $m+1$ in the basis $\{e_i:i\in[n]\}\cup\{e_\ast,e_{\overline\ast}\}$
where $e_\ast=(\underline 0,0,1)$ and $e_{\overline\ast}=(\underline 0,1,0)$.
The Pl\"ucker coordinates of each of $F^{m-1}$, $F^m$, $F^{m+1}$, and~$G$
appear as subcollections of those of~$L$ up to a multiplicative scalar,
namely where the index set varies among all sets
either containing $\ast$ or not, as appropriate,
and either containing $\overline\ast$ or not, as appropriate.
By choosing $c$, $c'$, and a global multiplicative scalar suitably,
one gets explicit representatives $(p_A)$ of the Pl\"ucker coordinates for~$L$
whose valuations agree with
$t_{i(j,m-1),m-1}(M)$, $t_{i(j,m),m}(M)$, and $t_{i(j,m+1),m+1}(M)$
on the $F^{m-1}$ coordinates, the $F^m$ coordinates, and the $F^{m+1}$ coordinates,
respectively.  Set $t_{i(j+1,m),m}(M)$ to consist of the remaining valuations,
those for the $G$ coordinates.
Since $L$ is a linear space, these valuations satisfy the tropical
Pl\"ucker relations, and hence $M$ satisfies property (D2$'$),
which is axiom (T2) of the axiom system in~\ref{prop:t}.

Of the remaining axioms, (T0) and~(T1), which are equivalent to
(D1$'$), (D1$''$), and (D0$''$), hold of~$M$ by the assumption
that the valuations of the Pl\"ucker coordinates of~$\mathcal F_*$ lie in $\mathcal C$.
Finally, the property (TS) holds of~$M$ by the initial stage of its construction.
So $M$ is a matroid over~$R$.
\end{proof}

\begin{proof}[Proof of Theorem~\ref{thm:BS}.]
The set $Z_w'$ can be taken to be $\im\pi\cup\Trop Z_w$:
this is not the ``right'' choice, as the makeup of this $Z_w'$
is rather different inside and outside $\mathcal C$, but it will do for the theorem statement.
The cone of the theorem is $\mathcal C$.

Proposition~\ref{prop:moduli lower bound}
establishes that $Z_w'\cap\mathcal C = \im\pi$
and gives the containment $\Trop Z_w\subseteq Z_w'$.
The containment $Z_w'\subseteq ZD_w$,
follows from the two containments $\Trop Z_w\subseteq ZD_w$ and $\im\pi\subseteq ZD_w$.
The first of these
is a consequence of the fundamental theorem of tropical geometry~\cite[\S3.2]{MS}
since the relations defining $ZD_w$ are tropicalisations of some of
the Pl\"ucker relations defining $Z_w$.
Lemma~\ref{lem:closure of finite} implies the second.

Finally, $\pi$ factors as $p_1\circ\widehat\pi$,
where $p_1$ is the map discarding the latter two components of the triples output by $\widehat\pi$.
The fibres of $\widehat\pi$ have dimension at most $D(k,\ell)$ per Proposition~\ref{prop:pi}.
Those of~$p_1$ are of dimension at most $k+\ell$,
as it takes $k$ parameters to specify $M(\emptyset)$ and $\ell$ to specify $M([n])$.
\end{proof}

\subsection{Further observations}

Neither of the inclusions bounding $Z'_w$ in Theorem~\ref{thm:BS} is strict.
The assertion $Z'_w\subsetneq ZD_w$ follows from the fact that
the map $\pi$ of Proposition~\ref{prop:pi} is not surjective.
This failure of surjectivity is due to the fact that property (D2$'$)
does not follow from the properties (D4), (D3), and~(D3$'$)
imposed by tropical incidence conditions
(nor from the inequalities of $\mathcal C$).
It would be enough to get the equality $Z'_w=ZD_w$ if (D2$'$) etc.\ were implied by
(D4), (D3), and~(D3$'$) up to altering one $t_{r,i}$ by an additive scalar,
which is what is true classically;
this is the best that could be hoped for anyway, on the account that
(D4), (D3), and~(D3$'$) are invariant under additive actions of tropical tori which alter (D2$'$).
But even this is false.

\begin{example}
Let $n=4$.  We will define the quantities $t_i(A)$ for the pairs
$(|A|,i) = (2,0)$, $(3,0)$, $(1,1)$, and $(2,1)$.
Let these all equal zero except for
$t_0(12) = 2$, $t_1(12) = 1$, $t_0(34) = 1$,
and $t_0(A)=\lambda$ for all $|A|=3$, where $\lambda$ is a parameter.
Regardless of the choice of~$\lambda$, property (D2$'$) fails.
This is because for the minimum to be attained twice among
\[t_0(123)+t_1(1)=\lambda, t_1(12)+t_0(13)=1, t_0(12)+t_1(13)=2\]
we require $\lambda=1$, while among
\[t_0(234)+t_1(4)=\lambda, t_1(34)+t_0(24)=0, t_0(34)+t_1(24)=1\]
we require $\lambda=0$.
However, properties (D4) and (D3) and (D3$'$) hold,
as can be checked: indeed, in every case, the repeated minimum is either
$0+0$ or~$0+\lambda$.
\end{example}

Next we show that every tropical linear space arises
as $t_{r,0}$ from some matroid over~$R$.
Together with the existence of nonrealisable tropical linear spaces,
this implies that $\Trop Z_w\subsetneq Z'_w$.

\begin{proposition}\label{prop:section of pi}
Let $p$ be a tropical Pl\"ucker vector of rank~$r$ on~$[n]$, the minimum of whose entries equals zero.
Then there is a matroid $M$ over~$R$ of generic rank~$r$ on~$[n]$ with $t_{r,0}(M)=p$;
in fact, there is one in $\Mat{R}(r,n;0,0)$.
\end{proposition}

We leave open the obvious extension to flags:
\begin{question}
Does every flag of tropical linear spaces in $\Trop\P^{n-1}$
occur as the flag of spaces $L(t_{s,\max{r-s,0}}(M))$
for some matroid $M$ over~$R$?
\end{question}

The key constructions are \emph{stable sum} of tropical linear spaces
and its dual operation \emph{stable intersection}.
If $p$ and $p'$ are tropical Pl\"ucker vectors of respective ranks
$r$ and~$r'$ on the same set $[n]$, their stable sum is the Pl\"ucker vector
$p+p'$ of rank $r+r'$ given by
\[(p+p')_A = \min\{p_B+p'_C : B\mathbin{\dot\cup}C = A\}.\]
In the same setup, the stable intersection of $p$ and~$p'$
is the Pl\"ucker vector $p\cap p'$ of rank $r+r'-n$ given by
\[(p\cap p')_A = \min\{p_B+p'_C : ([n]\setminus B)\mathbin{\dot\cup}([n]\setminus C) = [n]\setminus A\}.\]
If $L(p)$ and $L(p')$ are tropicalisations of linear spaces over a field
in generic position, then $L(p+p')$ is the tropicalisation of their sum
and $L(p\cap p')$ of their union.
Thus we see geometrically that, if $L(p')$ is contained in $L(p'')$,
the operations $p\mapsto p+p'$ and $p\mapsto p\cap p''$ commute.

\newcommand{\lplus}{\mathop{\underline+}}
\newcommand{\lcap}{\mathop{\underline\cap}}
If $g$ is a tropical Pl\"ucker vector of rank~1 and $p$ any other tropical
Pl\"ucker vector on the same ground set $[n]$, of rank say $r$,
then let $p\lplus g$ be the tropical Pl\"ucker vector
of rank $r+1$ on $[n]\cup\{*\}$ with
\[(p\lplus g)_A = \begin{cases}
(p+g)_A & *\not\in A\\
p_{A\setminus\{*\}} & *\in A.
\end{cases}\]
Similarly, if $h$ is a tropical Pl\"ucker vector of rank~$n-1$,
let $p\lcap h$ be the tropical Pl\"ucker vector of rank $r$ on $[n]\cup\{\bar *\}$ with
\[(p\lcap h)_A = \begin{cases}
p_A & \bar *\not\in A\\
(p\cap h)_{A\setminus\{\bar *\}} & \bar *\in A.
\end{cases}\]

\begin{lemma}\label{lem:section of pi}
Let $g$ and $h$ be tropical Pl\"ucker vectors of respective ranks
1 and $n-1$ on~$[n]$ such that $L(g)\subseteq L(h)$.
Let $v\in\pm\val R\cup\{\infty\}$.
Define Pl\"ucker vectors $g'$ of rank 1 on~$[n]\cup\{\bar*\}$
by $g'_{\{\bar*\}} = v$ while $g'_{A} = g_{A}$ for other $A$,
and $h$ of rank $n$ on $[n]\cup\{*\}$
by $h'_{[n]} = v$ while $h'_{A} = h_{A\setminus\{*\}}$ for other $A$.
Then
\[(p\lplus g)\lcap h' = (p\lcap h)\lplus g'\]
for any tropical Pl\"ucker vector $p$.
\end{lemma}

\begin{proof}
On expanding the definitions of the vectors on either side,
all components are manifestly equal except those indexed by sets
that contain $\bar*$ but not $*$.  If $A$ is such a set, then
\begin{align*}
((p\lplus g)\lcap h')_A &= ((p\lplus g)\cap h')_{A\setminus\{\bar*\}}
\\&=\min\{\min\{(p\lplus g)_{A\cup\{b\}\setminus\{\bar*\}}+h'_{[n]\cup\{*\}\setminus\{b\}}
: b\in [n]\setminus A\},
\\&\qquad (p\lplus g)_{A\cup\{*\}\setminus\{\bar*\}}+h'_{[n]}\}
\\&=\min\{((p+g)\cap h)_{A\setminus\{\bar*\}}, p_{A\setminus\{\bar*\}}+v\}
\end{align*}
while by an analogous computation
\[((p\lcap h)\lplus g')_A =
\min\{((p\cap h)+g)_{A\setminus\{\bar*\}}, p_{A\setminus\{\bar*\}}+v\}.\]
The first terms of the minima are equal by the assumption $L(g)\subseteq L(h)$.
\end{proof}


\begin{proof}[Proof of Proposition~\ref{prop:section of pi}]
Write $0$ for the tropical Pl\"ucker vector of rank 1 on any set
all of whose coordinates are~0, and $0^*$ for the tropical Pl\"ucker vector
of rank $|E|-1$ on any set $E$ all of whose coordinates are~0.
Then $0' = 0$ and $(0^*)' = 0^*$ in the notation of Lemma~\ref{lem:section of pi},
so that, by the lemma,
\[q = p
\underbrace{\mbox{}\lcap 0^*\cdots\lcap 0^*}_r
\underbrace{\mbox{}\lplus 0\cdots\lplus 0}_{n-r}\]
is a well-defined tropical Pl\"ucker vector of rank $n$ on $2n$~elements.
We claim that it is the image of the requisite matroid over~$R$
under the embedding $\xi$ of Proposition~\ref{prop:slice},
where $G$ consists of all the elements $\bar*$ and $Z$ of all the elements~$*$.

Indeed, on expanding the definitions, we find
\[q_A = \min\{p_B : |B|=r, |B\cap A|\geq r-|G\cap A|\}\]
which is clearly independent of any finer information about $G\cap A$ or $Z\cap A$
than their cardinalities.  Moreover, if $|Z\cap A|=0$ or $|G\cap A| = r = |G|$, then
$q_A$ is the minimum of all the coordinates of~$p$, so it is zero by assumption.
This shows that $q$ lies in the linear space called $W$ above, so it is in the
image of~$\xi$.
\end{proof}


\appendix
\section{Proofs from Section~\ref{sec2}}\label{app:d}

We see no better way to prove Lemma~\ref{lem:d}
than repeating the proofs from \cite{FM} in the broader setting of a valuation that is not necessarily discrete.

\begin{proof}[Proof of Proposition~\ref{lem:d}]
Condition~(L0) is simply the condition that the data $d_\ell(M(A))$ are the invariants of some $R$-module.
With that handled, we must show necessity and sufficiency of the conditions (L1), (L2a), (L2b).

\noindent\emph{Necessity.} Property (L1) is immediate with our definition of the invariants~$t_i$.
There is some $x\in M(A)$ so that we may identify $M(Ab)=M(A)/\langle x\rangle$.
Then clearly, for any $i$,
\begin{align*}
       \min\length(N/\langle z_1,\ldots,z_i\rangle)
  &\geq\min\length(N/\langle z_1,\ldots,z_i,x\rangle)
\\&\geq\min\length(N/\langle z_1,\ldots,z_i,z_{i+1}\rangle),
\end{align*}
all minima being over the choices of the elements~$z_j$.

For the two-element properties,
let $M'$ be $M\otimes_R R/I_\ell$.
Choose $x,y\in M'(A)$ so as to make the identifications
$M'(Ab)=M'(A)/\langle x\rangle$,
$M'(Ac)=M'(A)/\langle y\rangle$,
$M'(Abc)=M'(A)/\langle x,y\rangle$.
Let $\overline x$ be the image of $x$ in~$M'(Ac)$,
and $\overline y$ the image of $y$ in~$M'(Ab)$.
Then all rows and columns of the following diagram are exact.
\[
\xymatrix{
  & 0\ar[d] & 0\ar[d] & 0\ar[d] \\
0\ar[r] & Rx\cap Ry\ar[d]\ar[r] & Ry\ar[d]\ar[r] & R\overline y\ar[d]\ar[r] & 0 \\
0\ar[r] & Rx\ar[d]\ar[r] & M'(A)\ar[d]\ar[r] & M'(Ab)\ar[d]\ar[r] & 0 \\
0\ar[r] & R\overline x\ar[d]\ar[r] & M'(Ac)\ar[d]\ar[r] & M'(Abc)\ar[d]\ar[r] & 0 \\
  & 0 & 0 & 0
}
\]
The relations these exact sequences induce
in the Grothendieck group of finitely presented $R$-modules show that
\begin{align}
&\mathrel{\phantom=}
d_{\leq\ell}(M(A)) - d_{\leq\ell}(M(Ab)) - d_{\leq\ell}(M(Ac)) + d_{\leq\ell}(M(Abc))\notag
\\&= [M'(A)] - [M'(Ab)] - [M'(Ac)] + [M'(Abc)]\notag
\\&= [Rx] - [R\overline x]\notag
\\&= [Rx\cap Ry] \geq 0,\label{eq:L2a nec}
\end{align}
where brackets denote class in the Grothendieck group,
given the order it is endowed with by its identification with~$\val(R)$.
This proves~(L2a).

Now, if $N$ is a finitely presented $R$-module and $\ell\in\val(R)^+$,
the analysis before Proposition~\ref{prop:t} using structure theory
of finitely presented $R$-modules
shows that $d_\ell(N)$ is the number of free summands in~$N\otimes_R R/I_\ell$.
This is also the number of generators of $J(N\otimes_R R/I_\ell)$,
where $J$ is a nonzero ideal chosen small enough to annihilate every
nonfree summand of~$N$.

If the inequality \eqref{eq:L2a nec} is strict, then $Rx\cap Ry$ is nontrivial,
that is, there is a relation $rx=sy\neq0$ in $M'(A)$.
Taking the $N$ of the previous paragraph to be $M'(A)$,
and $J$ furthermore small enough that $rx\not\in\mathfrak mJM'(A)$,
we have that $JM'(A)\cap Rx$ and $JM'(A)\cap Ry$ are equal,
namely, both are equal to $JM'(A)\cap Rrx$.
Therefore $JM'(A)/(JM'(A)\cap Rx) \cong JM'(Ab)$
is isomorphic to $JM'(A)/(JM'(A)\cap Ry) \cong JM'(Ac)$,
which entails $d_\ell(M(Ab)) = d_\ell(M(Ac))$.  This proves~(L2b).

\noindent\emph{Sufficiency.} Sufficiency requires proving that
for each $A\subseteq E$ and $b,c\not\in A$,
elements $x,y\in M(A)$ can be constructed so that the quotients have the correct isomorphism classes.
For want of a  way to lay hands on these with any insight,
we reproduce versions of the explicit formulae given in~\cite{FM} in terms of generators
usable in arbitrary valuation rings.

By~(L1), for any of the four pairs $(A',A'b')$ of sets
$(A,Ab)$, $(A,Ac)$, $(Ab,Abc)$, and $(Ac,Abc)$,
the difference $d_\ell(A',b') := d_\ell(A') - d_\ell(A'b')$ is drawn from $\{0,1\}$ for any $\ell\in\val(R)^+$.
As $\ell$ varies, the following is true of the quantities $d_\ell(A')$, and thus their differences:
\begin{itemize}
\item their value changes finitely often;
\item they are continuous from below, in the sense that $d_\ell(A')$ is equal to~$d_{\ell'}(A')$
for all $\ell'<\ell$ sufficiently close to $\ell$ (in the order topology).
\end{itemize}
Let $I(A',b')$ be the finite set of $\ell\in\val(R)^+$ such that
$d_\ell(A',b')=1$ but
$d_{\ell'}(A',b')=0$ for all $\ell'>\ell$ sufficiently close to $\ell$,
together with $\infty$ if $d_\ell(A',b')$ stabilises at~$1$ for sufficiently large $\ell$.
We are assured of the existence of various cyclic summands of $M(A)$
all participating in a single direct sum decomposition
for which we may choose generators as follows:
a generator $e_i$ with $\langle e_\ell\rangle = R/I_\ell$ for each $\ell\in I(A,b)\cup I(A,c)$,
and a second generator $\varepsilon_\ell$ distinct from~$e_\ell$
with $\langle\varepsilon_\ell\rangle = R/I_\ell$ for each $\ell\in I(Ab,c)\cap I(Ac,b)$.

Define two elements $x$ and~$y$ in $M(\emptyset)$ by
\begin{align*}
x &= \sum_{\ell\in I(A,b)} \tau(\ell-d_{\leq\ell}(A,b))\, e_\ell \\
y &= \sum_{\ell\in I(A,c)} \tau(\ell-d_{\leq\ell}(A,c))\, e_\ell \\
&\qquad + \!\!\!\!\!\!\!\!\sum_{\ell\in I(Ab,c)\setminus \big(I(A,c) \setminus I(A,b)\big)}\!\!\!\!\!\!\!\!
\tau(\ell-d_{\leq\ell}(Ab,c))\, \varepsilon_\ell.
\end{align*}
In the above formulae, the $\tau(\ell)$ are elements of~$R$
suitably chosen to replace powers of a uniformising parameter in the formulae for the DVR case.
We let $Z\subseteq\val(R)$ be the finite set of all expressions
$\ell-d_{\leq\ell}(A',b')$ for $\ell\in I(A',b')$, as well as all differences between these.
(This contains the set of all exponents that appear on a uniformising parameter in the proofs in~\cite{FM}.)
Then we choose $\tau(\ell)\in R$ for all $\ell\in Z$ such that $\val(\tau(\ell)) = \ell$,
and $\tau(\ell)\tau(m) = \tau(\ell+m)$ whenever $\ell$, $m$ and $\ell+m$ are all elements of~$Z$.
This can be arranged by choosing a basis $B$ of the $\mathbb Q$-vector space generated by~$Z$,
scaled so that every element of~$Z$ is an integral combination of elements of~$B$,
and letting the $\tau(\ell)$ be Laurent monomials in elements of~$\Frac R$ whose valuations are the elements of~$B$.

The reader can check that
$M(A)/\langle x\rangle\cong M(Ab)$,
$M(A)/\langle y\rangle\cong M(Ac)$, and
$M(A)/\langle x,y\rangle\cong M(Abc)$
by following the proofs in \cite[Propositions 5.2, 5.4]{FM}
\emph{mutatis mutandis}.
They amount to the explicit construction of new generating sets for each of these three quotients
in terms of linear combinations of the given generators.
\end{proof}

Lemma~\ref{lem:d} being established, the task that remains in the proof of Proposition~\ref{prop:t}
is to translate from the $d$s to the $t$s.
It is this that we take up next.

\begin{proof}[Proof of Proposition~\ref{prop:t}]
For concision we'll write $d_\ell(A)$ for $d_\ell(M(A))$ in this proof.
The invariants $t_i(A)$ and $d_\ell(A)$ determine one another by the rule
\begin{equation}\label{eq:d<=>t}
t_i(A) - t_{i+1}(A) \geq \ell \Longleftrightarrow d_\ell(A) > i.
\end{equation}
Given the $t_i(A)$, this rule determines the $d_\ell(A)$ completely,
while given the $d_\ell(A)$ the further information needed to fix the $t_i(A)$
is that $t_i(A)=0$ whenever $d_\ell(A)\leq i$ for all~$i$.  This implies (TS);
conversely, (TS) implies the bounded-above condition in~(L0).
Next, (T0) holds given the $d_\ell(A)$ because
\[t_i(A) - t_{i+1}(A) < \ell \leq t_{i+1}(A) - t_{i+2}(A)\]
is equivalent to
\[i\geq d_\ell(A) > i+1;\]
similarly the nonincreasing condition in~(L0) holds given the $t_i(A)$ because
\[d_\ell(A) \leq i < d_m(A)\]
is equivalent to
\[\ell > t_i(A) - t_{i+1}(A) \geq m.\]

(L1) says that if $d_\ell(Ab)=j$ then $d_\ell(A)$ is in the interval $[j,j+1]$.
Thus if $i$ is a (strict) lower bound for $d_\ell(Ab)$, it is also for $d_\ell(A)$;
while if $i+1$ is a lower bound for $d_\ell(A)$, then $i$ is for $d_\ell(Ab)$.
This is equivalent by \eqref{eq:d<=>t} to
any (weak) lower bound $\ell$ for $t_i(Ab) - t_{i+1}(Ab)$ also bounding $t_i(A) - t_{i+1}(A)$,
and any lower bound $\ell$ for $t_{i+1}(A) - t_{i+2}(A)$ also bounding $t_i(Ab) - t_{i+1}(Ab)$,
which is (T1).

It remains to show that, given the preceding axioms, (T2) is equivalent to~(L2a,b).
Set $t_{-1}(Abc)=\infty$ by convention.
Assuming (T1), for any natural~$i$, we have
\begin{multline}\label{eq:dt strict}
\max\{t_{i+1}(A)-t_{i+2}(A),t_i(Abc)-t_{i+1}(Abc)\}\leq \\
\min\{t_i(Ab)-t_{i+1}(Ab),t_i(Ac)-t_{i+1}(Ac)\}=:\ell.
\end{multline}
Let's suppose for a start that this inequality is strict.
Then \eqref{eq:d<=>t} says that $d_\ell(A)\leq i+1$ and $d_\ell(Abc)\leq i$ directly,
and together with a further application of~(T1) to the left hand side
it says $d_\ell(Ab)\leq i+1$ and $d_\ell(Ac)\leq i+1$.
Turning then to the right hand side,
\eqref{eq:d<=>t} says that all of these inequalities are strict,
that is, $d_\ell(A) = d_\ell(Ab) = d_\ell(Ac) = i+1$ and $d_\ell(Abc) = i$.

We observe that, for any set $A'$,
\[d_{\leq\ell}(A') = t_i(A') + j\ell\]
where $j$ is the minimal natural such that $t_j(A')-t_{j+1}(A')<\ell$, i.e.\ $j=d_\ell(A')$.
Assuming (L1), therefore, we have
\[d_{\leq\ell}(A') - d_{\leq\ell}(A'b') = \begin{cases}
t_j(A') - t_j(A'b') & \mbox{if } d_\ell(A') = d_\ell(A'b') = j \\
t_{j+1}(A') - t_j(A'b') + \ell & \mbox{if } d_\ell(A') = j+1, d_\ell(A'b') = j.
\end{cases}\]
Using this when $i$ and~$\ell$ are as chosen in the previous paragraph,
(L2a,b) translates exactly into~(T2).
If equality holds in~\eqref{eq:dt strict},
then we can only conclude $d_\ell(A)\geq i+1$ and $d_\ell(Abc)\geq i$,
and the translation might yield subscripts in $t_\bullet(A)$ and $t_\bullet(Abc)$ that are too large.
But we have $t_j(A')\geq t_{j+1}(A')$ for any $j$ and~$A'$,
so the translation of (L2a,b) in this event is a stronger inequality than~(T2).
In any case (L2a,b) implies~(T2).

Conversely, we must show that (T2) implies (L2a) and~(L2b).
Of these (L2b) is implied straightaway by the corresponding proviso on equality in~(T2) using \eqref{eq:d<=>t},
but (L2a) will require a longer argument.
As above, (T2) is the translation of~(L2a) into a relation on the quantities $t_i(A)$
in the case that $\ell$ is chosen according to~\eqref{eq:dt strict} for some~$i$.
So we must handle the values of $\ell$ not of this form.

If $\ell$ takes a value strictly between the left and right sides of~\eqref{eq:dt strict},
then the translation of (L2a) has the same left side and a lesser right side, so (T2) implies it.
Next suppose $\ell$ exceeds $\min\{t_i(Ab)-t_{i+1}(Ab),t_i(Ac)-t_{i+1}(Ac)\}$
but does not exceed $\min\{t_i(A)-t_{i+1}(A),t_{i-1}(Abc)-t_i(Abc)\}$.
According to how $\ell$ compares to $t_i(Ab)-t_{i+1}(Ab)$ and $t_i(Ac)-t_{i+1}(Ac)$,
there are three forms (L2a) might take:
\begin{align}\label{eq:dt sides}
t_{i+1}(A)-t_i(Ab)-t_{i+1}(Ac)+t_i(Abc) &\geq 0; \\\notag
t_{i+1}(A)-t_{i+1}(Ab)-t_i(Ac)+t_i(Abc) &\geq 0; \\\notag
t_{i+1}(A)-t_i(Ab)-t_i(Ac)+t_i(Abc) &\geq -\ell.
\end{align}
The first of these follows by subtracting $t_i(Ab)-t_{i+1}(Ab)$ from both sides of~(T2);
the minimum on the right side resolves to zero because if $i=d_\ell(Ab)$ while $i+1=d_\ell(Ac)$
then necessarily $t_i(Ab)-t_{i+1}(Ab)<t_i(Ac)-t_{i+1}(Ac)$.
The same argument with the roles of $b$ and~$c$ reversed establishes the second.
For the third case, we subtract $t_i(Ab)-t_{i+1}(Ab) + t_i(Ac)-t_{i+1}(Ac)$ from~(T2),
converting the right side to
\begin{multline*}
\min\{-t_i(Ac)+t_{i+1}(Ac),-t_i(Ab)+t_{i+1}(Ab)\} \\= -\max\{t_i(Ab)-t_{i+1}(Ab),t_i(Ac)-t_{i+1}(Ac)\};
\end{multline*}
but because $i=d_\ell(Ab)=d_\ell(Ac)$ we have $\ell>\max\{t_i(Ab)-t_{i+1}(Ab), t_i(Ac)-t_{i+1}(Ac)\}$
so that the desired right side is a weakening of the one thus obtained.

The remaining two forms which (L2a) might take, up to choice of~$i$,
correspond to values of~$\ell$ exceeding the minimum of
$t_i(A)-t_{i+1}(A)$ and $t_{i-1}(Abc)-t_i(Abc)$ but not exceeding their maximum.  These are
\begin{align}\label{eq:dt tips}
t_i(A)-t_i(Ab)-t_i(Ac)+t_i(Abc) &\geq 0;\\\notag
t_{i+1}(A)-t_i(Ab)-t_i(Ac)+t_{i-1}(Abc) &\geq 0.
\end{align}
The argument of the preceding paragraph shows that (T2) implies
one of the first two clauses of~\eqref{eq:dt sides}
(independently of concerns about which of these is the translation of~(L2a)).
Reversing the roles of $b$ and~$c$ if needed, we may assume without loss of generality that it is the first.
But (T1) implies that
\begin{align*}
t_i(A)-t_i(Ac) &\geq t_{i+1}(A)-t_{i+1}(Ac),\\
t_{i+2}(A)-t_{i+1}(Ab) &\geq t_{i+1}(A)-t_i(Ab)
\end{align*}
and these, together with the first clause of~\eqref{eq:dt sides},
imply the two inequalities of~\eqref{eq:dt tips},
the latter after substituting $i-1$ for~$i$.
We conclude that (L2a) is a consequence of the (T) conditions, and the proof is finished.
\end{proof}

\end{document}